\documentclass[reqno]{amsart}
\usepackage{euscript,amsmath,amssymb,amsbsy}
\usepackage[arrow,matrix,curve,rotate]{xy}
\usepackage{a4wide,graphicx,longtable}

\newtheorem{stuff}{Stuff}[section]
\newtheorem{theorem}[stuff]{\bf Theorem}
\newtheorem{proposition}[stuff]{\bf Proposition}
\newtheorem{lemma}[stuff]{\bf Lemma}
\newtheorem{corollary}[stuff]{\bf Corollary}

\newenvironment{definition}{%
\vskip1ex\refstepcounter{stuff}\trivlist \itemindent 0pt
\item[\hskip\labelsep\bf Definition \thestuff.]%
\ignorespaces}{\endtrivlist\vskip1ex}%
\newenvironment{remark}{%
\vskip1ex\refstepcounter{stuff}\trivlist \itemindent 0pt
\item[\hskip\labelsep\bf Remark \thestuff.]%
\ignorespaces}{\endtrivlist\vskip1ex}%
\newenvironment{notation}{%
\vskip1ex\refstepcounter{stuff}\trivlist \itemindent 0pt
\item[\hskip\labelsep\bf Notation \thestuff.]%
\ignorespaces}{\endtrivlist\vskip1ex}%

\newenvironment{thm-nono}{
\vskip1ex\trivlist \itemindent 0pt
\item[\hskip\labelsep\bf Theorem.]%
\it\ignorespaces}{\endtrivlist\vskip1ex}%

\newtheorem{s-theorem}[sstuff]{\bf Theorem}
\newtheorem{s-proposition}[sstuff]{\bf Proposition}

\let\rar\rightarrow
\let\lar\longrightarrow

\let\hra\hookrightarrow
\let\mt\mapsto

\font\tenmsa=msam10 %
\newcommand\hdashpiece{%
{\vrule height2.75pt depth-2.35pt width2.3pt \kern1.7pt}}%
\newcommand\hdashpieces{%
{\hdashpiece\hdashpiece\hdashpiece\hdashpiece}}%
\newcommand\dashto{\mathrel{%
\hdashpiece\hdashpiece\kern-0.4pt\hbox{\tenmsa K}}}%
\newcommand\dashar{\mathrel{%
\hdashpieces\kern-0.4pt\hbox{\tenmsa K}}}%

\let\cal\mathcal
\let\mbb\mathbb

\let\mfrak\mathfrak
\let\bsymb\boldsymbol 
\DeclareFontFamily{OT1}{rsfs}{}
\DeclareFontShape{OT1}{rsfs}{n}{it}{<->rsfs10}{}
\DeclareMathAlphabet{\crl}{OT1}{rsfs}{n}{it}

\let\ovl\overline

\let\tld\tilde
\let\nit\noindent
\let\disp\displaystyle
\let\srel\stackrel
\let\vphi\varphi

\let\veps\varepsilon
\newcommand\rd{{\rm d}} 

\newcommand\Aut{\mathop{\rm Aut}\nolimits}
\newcommand\End{\mathop{\rm End}\nolimits}
\newcommand\Ext{\mathop{\rm Ext}\nolimits}
\newcommand\Hom{\mathop{\rm Hom}\nolimits}

\newcommand\Ker{{\rm Ker}}
\newcommand\Coker{{\rm Coker}}
\newcommand\Img{{\rm Im}}
\newcommand\NS{{\rm NS}}
\newcommand\rk{{\rm rk}}

\numberwithin{equation}{section}

\let\l\lambda
\let\L\Lambda

\let\O\Omega

\let\si\sigma

\let\th\theta
\let\Th\Theta
\let\sm\setminus

\newcommand\bone{{1\kern-0.57ex\rm l}}

\newcommand{\ouset}[3]{\hbox{$\overset{#2}{\underset{#1}{#3}}$}}

\newcommand\cF{{\cal F}}
\newcommand\cG{{\cal G}}

\newcommand\cO{{\cal O}}

\newcommand\T{\mathop{\sf T\kern-.2ex}\nolimits}
\newcommand\N{\mathop{\sf N\kern-.1ex}\nolimits}

\newcommand\pr{\mathop{\rm pr}\nolimits}
\newcommand\Sym{\mathop{\rm Sym}\nolimits}

\newcommand\codim{{\rm codim}}
\newcommand\Gl{{\rm Gl}}
\newcommand\PGl{{\rm PGl}}

\newcommand\ch{{\rm ch}}
\renewcommand\det{{\rm det}}
\newcommand\rel{{\rm rel}}
\let\ges\geqslant
\let\les\leqslant

\newcommand\mx{{\rm max}}

\newcommand\surj{{\rar\kern-1.85ex\rar}}
\newcommand\fbr{{\Phi}}

\newcommand\hn[2]{\cal #1^{L_{#2}\text{\rm-HN}}}
\newcommand\hnr[1]{\cal#1^{L\text{\rm-rel-HN}}}

\newcommand\ad{{\rm ad}}
\newcommand\al[1]{{\bigl[#1\bigr]}_{AL}}
\newcommand\LL{{L}}
\newcommand\invq{{/\!/\,}}
\newcommand\Supp{\mathop{\rm Supp}\nolimits}
\newcommand\Spec{\mathop{\rm Spec}\nolimits}
\newcommand\Tor{\mathop{\rm Tor}\nolimits}
\newcommand\ext{\mathop{\rm ext}\nolimits}
\newcommand{\vb}{{\rm vb}}
\newcommand\Hilb{\mathop{\rm Hilb}\nolimits}
\newcommand{\lft}{{\rm left}}
\newcommand{\rgt}{{\rm right}}

\title[Semi-stable vector bundles on fibred varieties]%
{Semi-stable vector bundles on fibred varieties}
\author{Mihai Halic}


\address{
\begin{minipage}{30em}
\end{minipage}\vspace{1ex}
}

\begin{document}

\begin{abstract}
Let $\pi\,{:}\,Y{\to}X$ be a surjective morphism between two irreducible, smooth 
complex projective varieties with $\dim Y{>}\dim X{>}0$. We consider polarizations 
of the form $\LL_c=\LL+c\cdot\pi^*A$ on $Y$, with $c>0$, where $\LL,A$  are ample 
line bundles on $Y,X$ respectively. 

For $c$ sufficiently large, we show that the restriction of a torsion free sheaf $\cF$ 
on $Y$ to the generic fibre $\fbr$ of $\pi$ is semi-stable as soon as $\cF$ is 
$\LL_c$-semi-stable; conversely, if $\cF\otimes\cO_\fbr$ is $L$-stable on $\fbr$, 
then $\cF$ is $\LL_c$-stable. We obtain explicit lower bounds for $c$ satisfying 
these properties. Using this result, we discuss the construction of semi-stable 
vector bundles on Hirzebruch surfaces and on $\mbb P^2$-bundles over $\mbb P^1$, 
and establish the irreducibility and the rationality of the corresponding moduli spaces.
\end{abstract}

\maketitle

\section*{Introduction}

It is a non-trivial problem to explicitly exhibit (semi-)stable vector bundles in higher 
dimensions, and to study the geometric properties of the corresponding moduli spaces; 
these latter are mostly obtained as geometric invariant quotients of (large) quot schemes 
(see \cite{ma0,ma,si,la,hl}). 
Stable vector bundles of rank exceeding the dimension of the base, 
with large second Chern class are constructed in \cite[Appendix]{ma}. 
Other higher dimensional examples include the instanton bundles \cite{os,jar}, 
which generalize the well-known ADHM construction \cite{adhm,bh}. 
Also, the construction of instantons on $\mbb P^3$ was extended in \cite{ku} to Fano 
threefolds of index two, with cyclic Picard group. 

This article attempts to develop yet another method of constructing (semi-)stable sheaves. 
We investigate the relationship between the (semi-)stability of a sheaf on the total space 
of a fibre bundle, and the (semi-)stability of its restriction to the {\em generic} fibre. 
This is different from the relative semi-stability concept in \cite{ma,si}, where one 
requires that the restriction to {\em each} geometric fibre is semi-stable. 
Let $\pi:Y\to X$ be a surjective morphism between two irreducible, smooth complex 
projective varieties, with $d_Y{:=}\dim Y>d_X{:=}\dim X>0$. 
Such a $\pi$ will be called \emph{a fibration}. 
Let $A$ be an ample line bundle on $X$, and $L$ be a big, semi-ample (that is some 
power is globally generated), and $\pi$-ample line bundle on $Y$. For $c>0$, we denote 
$\LL_c:=L+c\cdot\pi^*A$, and define the slope of a torsion free sheaf $\cG$ on $Y$ 
with respect to $L_c,L,A$ by the formula  
$$
\mu_{\LL_c}(\cG):=
\frac{c_1(\cG)\LL_c A^{d_X-1}L^{d_Y-d_X-1}}{\text{rank}(\cG)}.
$$ 
The definition is inspired from \cite[pp. 260]{la}, which considers semi-stability with 
respect to a collection of nef divisors. One can interpret $\mu_{L_c}$ as the slope of 
the restriction of $\cG$ to a general (movable) curve cut out by (multiples of) $L_c,L,A$. 
This ties in with \cite{gt}, where is argued that in higher dimensions one should consider 
`polarizations' with respect to movable curves, rather than ample divisors. 
If $X$ is a curve, the formula coincides, after replacing $c$ by $(d_Y-1)c$, with the usual 
slope with respect to $L_c$. Moreover, regardless of $d_X$, $(\LL_c,L,A)$-semi-stability 
implies usual $\LL_c$-semi-stability, and, conversely, usual $\LL_c$-stability implies 
$(\LL_c,L,A)$-stability, for $c\gg0$. 

\begin{thm-nono}
Let $\cF$ be a torsion free sheaf of rank $r$ on $Y$. 
Then there is a constant $k_\cF$ such that the following hold: 
\begin{enumerate}
\item If $\cF$ is $\LL_c$-(semi-)stable with $c>k_\cF$, then the restriction of $\cF$ to 
the generic fibre of $\pi$ is semi-stable, and $\cF$ is $L_a$-(semi-)stable for all $a\ges c$.

\item If the restriction of $\cF$ to the generic fibre of $\pi$ is stable, 
then $\cF$ is $\LL_c$-stable for all $c>k_\cF$. 
\end{enumerate}
The same holds for principal $G$-bundles on $Y\!$, for connected, reductive, 
linear algebraic groups $G$. 
\end{thm-nono} 

Thus any $\LL_c$-(semi-)stable torsion free sheaf on $Y$, with $c\gg0$, determines 
a \emph{rational} map from $X$ to the (course) moduli space \cite{si} of $\pi$-relatively 
(semi-)stable sheaves on $Y$. Usually, this map \textit{does not extend} to $X$, 
which is the main difference to \emph{loc.~cit.} The result is a technical ingredient, 
a dimensional reduction, which is effective for varieties admitting fibrations onto lower 
dimensional varieties, such that one has \emph{a priori} knowledge about the 
(semi-)stable sheaves on the generic fibre. Polarizations of the form $\LL_c$, 
$c\gg0$, have been considered in \cite{fr,fmw} (and \cite{dng}) for vector 
(respectively principal) bundles on elliptically fibred surfaces, and in \cite{gl} 
on ruled surfaces. 

The theorem is proved in section \ref{sect:stab}, where we derive {\em two} distinct, 
{\em explicit} lower bounds for the constant $k_\cF$ above: they involve respectively 
the slope with respect to $L$ (see \ref{thm:main0}), and the discriminant of $\cF$ 
(see \ref{thm:main}). 
Our approach follows \cite[Theorem 5.3.2 and Remark 5.3.6]{hl}, where the result is 
proved for surfaces, and \cite[Section 3]{la}, where are developed higher dimensional 
techniques. For elliptically fibred surfaces, the result appears in \cite[Section 2]{fr}, 
\cite[Theorem 7.4]{fmw}. 

Sections \ref{sect:hirz} and \ref{sect:p21} illustrate the general principle with explicit 
examples. We study the moduli spaces of semi-stable vector bundles on 
\emph{Hirzebruch surfaces}, and on \emph{$\mbb P^2$-fibre bundles over $\mbb P^1$} 
respectively. Although each topic has its own intricacies, the underlying principle is the 
same: a semi-stable vector bundle on a fibration is a family of semi-stable vector bundles 
on the fibres. It is surprising that these topics are \emph{strongly connected}; 
for describing the geometry of vector bundles on $\mbb P^2$-fibrations over $\mbb P^1$, 
one needs to understand the case of Hirzebruch surfaces first. Thus, our approach yields a 
\emph{unified treatment}, and indeed generalizes \emph{several scattered subject matters}.
\medskip 

The \emph{former example}, that of stable vector bundles on Hirzebruch surfaces, 
was investigated in \cite{bu,nak,c-mr}, where the authors describe the geometry of the 
corresponding moduli spaces. 
We also mention \cite{ar}, for proving the non-emptiness and irreducibility of the moduli 
space of stable vector bundles with $c_1=0$ on a large class of rational surfaces, including 
the Hirzebruch surfaces.  
The last few years experienced a revived interest \cite{bbr} in constructing and 
understanding the properties of the moduli spaces of framed torsion free sheaves on 
Hirzebruch surfaces. 
Compared with the references above, we emphasize the \emph{brevity} and the 
\emph{detail} of our description of the geometry of the moduli space 
$\bar M^{L_c}_{Y_\ell}(r;0,n)$ of $L_c$-semi-stable rank $r$, torsion free sheaves on 
the Hirzebruch surface $Y_\ell$, with $c_1=0,c_2=n$. Theorem \ref{thm:hirz} reveals 
the existence of a stratification of $\bar M^{L_c}_{Y_\ell}(r;0,n)$ by locally closed 
strata, and the density of the stable vector bundles. Furthermore, we prove in \ref{thm:hilb} 
and \ref{thm:rtl} the existence of a surjective morphism onto 
${\rm Hilb}^n_{\mbb P^1}\cong\mbb P^n$, the Hilbert scheme of $n$ points on 
$\mbb P^1$. For $n=c_2=2$, the existence of this morphism is obtained in \cite{bu} 
by using monad theoretic techniques \cite{oss}, but is defined only for vector bundles. 
This morphism is the key for proving: 

\begin{thm-nono}\hskip-.5ex\text{\rm (see \ref{thm:rtl} and \ref{cor:P2}).}\hskip1ex 
$\bar M^{L_c}_{Y_\ell}(r;0,n)$ is rational, for any $n\ges r\ges 2$. 
Hence, for $\ell=1$, it follows that the moduli space $M_{\mbb P^2}(r;0,n)$ of stable 
rank $r$ vector bundles on $\mbb P^2$, with $c_1=0, c_2=n$ and $n\ges r$, is rational. 
\end{thm-nono}
The result should be compared with \cite{c-mr}, where is proved that 
$\bar M^{L_c}_{Y_\ell}(r;c_1,n)$ is rational for any $c_1$, under the assumption that 
the discriminant $2rn-(r-1)c_1^2$ is very large, but without giving any bounds. 
The rationality of $M_{\mbb P^2}(r;c_1,c_2)$ has been intensely studied over the past 
decades; see \cite{ba,ma3,es,mae} for $r=2$, \cite{kat,li} for $r=3$, \cite{yo,c-mr3} 
for arbitrary $r$. See also \cite{sch} for a quiver-theoretical approach. 
Most of these references prove rationality under some arithmetical restrictions on 
$r,c_1,c_2$. In our approach, we (almost) explicitly exhibit a rational variety which 
is birational to $M_{\mbb P^2}(r;0,n)$. 
\medskip

Our \emph{second example} concerns semi-stable vector bundles of arbitrary rank, 
with Chern classes $c_1=0, c_2=n\cdot[\cO_\pi(1)]^2, c_3=0$ on 
$Y_{a,b}=\mbb P\bigl(
\cO_{\mbb P^1}\oplus\cO_{\mbb P^1}(-a)\oplus\cO_{\mbb P^1}(-b)
\bigr)$. 
Here $0\les a\les b$ are two integers, 
so $\pi\,{:}\,Y_{a,b}\,{\to}\,\mbb P^1$ is a $\mbb P^2$-fibre bundle over $\mbb P^1$. 
Moduli spaces of \emph{rank two} vector bundles on $\mbb P^n$-bundles over curves 
were studied in \cite{c-mr2}, and more generally on Fano fibrations in \cite{nak2}, using 
extensions of rank one sheaves; thus the method is strongly adapted to the 
rank two case. 
In \ref{thm:p2-p1} we prove (as expected) that a semi-stable vector bundle on $Y_{a,b}$, 
satisfying some natural hypotheses, is the cohomology of a $1$-parameter family of 
monads on $\mbb P^2$. 
This \emph{generalizes} the construction of the \emph{instanton bundles} 
on $\mbb P^3$ trivialized along a line (see \cite{bh,do}). 

The next step is to investigate the geometric properties of the corresponding moduli space. 

\begin{thm-nono}\hskip-.5ex\text{\rm(see \ref{thm:MM}).}\hskip1ex
The moduli space of semi-stable vector bundles on $Y_{a,b}$ as above contains 
a non-empty `main component' which is irreducible, generically smooth, and rational.
\end{thm-nono}
The irreducibility of the full moduli space is a difficult issue, even in the case of 
$Y_{0,1}$, the blow-up of the $\mbb P^3$ along a line (see \cite{ti1,ti2}). 
This leads us to single out the main component of the full moduli space, 
in a similar vein to \cite{ti}. Let us remark that on threefolds, unlike for surfaces, 
the semi-stability and the Riemann-Roch formula are not sufficient to address the 
generic smoothness. 
Concerning the rationality issue, the author of this article could find only the reference 
\cite[Corollary 3.6]{c-mr2} dealing with the rationality of certain moduli spaces of 
rank two vector bundles on higher dimensional varieties.  
We prove that the main component is birational to the moduli space of framed vector 
bundles on a reducible surface (a wedge) obtained by glueing a plane and a Hirzebruch 
surface along a line. Then our conclusion follows from the results obtained before. 
Apparently, this is new even in the much studied case of $\mbb P^3$; 
the introduction of \cite{ti} mentions the rationality of the moduli space of instanton 
vector bundles, for $c_2=2,3,5$.

The results are stated for varieties defined over $\mbb C$. However, the usual base change 
arguments imply that they hold over any algebraically closed ground field of characteristic 
zero.


\section{Relative semi-stability for vector bundles}{\label{sect:stab}}

Let $\pi:Y\to X$ be a surjective morphism between irreducible, smooth, projective 
varieties with 
$$
d_Y:=\dim Y>d_X:=\dim X>0,
$$
and denote by $\fbr$ its generic fibre. We consider an ample line bundle $A$ on $X$, 
and a big, semi-ample, and $\pi$-ample line bundle $L$ on $Y$. For any $c>0$, 
we denote throughout the paper $\LL_c\,{:=}\,L+c\cdot\pi^*A$, and we let 
$\NS(Y)_{\mbb Q}$ be the Neron-Severi group of $Y$ with rational coefficients.
\smallskip 

\begin{definition}{\label{def:rel-ss}} 
\nit{\rm(i)} 
For a torsion free sheaf $\cG$ on $Y$, we denote 
$\xi_\cG:=\frac{c_1(\cG)}{\rk(\cG)}\in \NS(Y)_{\mbb Q}$, and define 
{\em the slope} of $\cG$ with respect to $\LL_c,\LL,A$ by the formula 
\begin{equation}\label{eq:Lc-ss}
\mu_{\LL_c}(\cG):=\xi_\cG\LL_c A^{d_X-1}L^{d_Y-d_X-1}
=\xi_\cG\cdot\bigl(
A^{d_X-1}L^{d_Y-d_X}+cA^{d_X}L^{d_Y-d_X-1}
\bigr).
\end{equation} 

\nit{\rm(ii)} 
The slope of a torsion free sheaf $\cG'$ on the generic fibre $\fbr$ is defined as 
$$
\mu_{L,\fbr}(\cG'):=\xi_{\cG'}\cdot A^{d_X}L^{d_Y-d_X-1}.
$$ 

\nit{\rm(iii)} 
The torsion free sheaf $\cF$ on $Y$ is {\em $\LL_c$-(semi-)stable} 
if for all saturated subsheaves $\cG\subset\cF$ holds  
\begin{equation}\label{eq:ss}
\mu_{\LL_c}(\cG)\underset{(\les)}{<}\mu_{\LL_c}(\cF).
\end{equation} 

\nit{\rm(iv)} 
We say that $\cF$ is {\em $\pi$-relatively $L$-(semi-)stable} 
if its restriction to the {\em generic} fibre of $Y\srel{\pi}{\rar} X$ 
is (semi-)stable with respect to $L\otimes\cO_\fbr$.
\end{definition}

\begin{notation}\label{eq:[x]} 
Let $\cG$ be a torsion free sheaf on $Y$. \\ 
\nit{\rm(i)} Henceforth we denote by $\cG_\Phi$ the restriction of $\cG$ 
to the generic fibre of $Y\to X$.
\smallskip 

\nit{\rm(ii)} For any $c>0$, we let $\hn{G}{c}$ be the maximal, saturated, 
$\LL_{c}$-de-semi-stabilizing subsheaf of $\cG$, that is the first term of 
its Harder-Narasimhan filtration with respect to $\LL_c$. We remark that, 
since $L$ is big and semi-ample on $Y$, one can still define $\hn{G}{}$, 
corresponding to $c=0$, by a limiting argument (see \cite[pp. 263]{la}).
\smallskip

\nit{\rm(iii)} Let $\hnr{G}$ be the (unique) maximal, saturated subsheaf of $\cG$, 
whose restriction to the generic fibre $\fbr$ is the first term of the Harder-Narasimhan 
filtration of $\cG_\fbr$ with respect to $L_\fbr$. It is defined as the sum of all the 
subsheaves $\cG'\subset\cG$ such that $\cG'_\fbr={(\cG_\fbr)}^{L_\fbr\text{-HN}}$. 
($\hnr{\cG}$ is called the first term of the $\pi$-relative Harder-Narasimhan 
filtration of $\cG$ with respect to the relatively ample line bundle $L$  
(See \cite[Section 2.3]{hl}).
\smallskip 

\nit{\rm(iv)} To save space, instead of the exact sequence 
$0{\to} A{\to} B{\to} C{\to} 0$ we will write $A{\subset} B\surj\,C$. 
\end{notation}

For a torsion free sheaf $\cF$ of rank $r$ on $Y$, we investigate the semi-stability of the 
restriction of $\cF$ to the generic fibre $\fbr$, given that $\cF$ is $L_c$-semi-stable. 
We prove that $\mu_{\LL_c}$-semi-stability implies $\pi$-relative semi-stability, 
and conversely, $\pi$-relative stability implies $\mu_{\LL_c}$-stability, for $c$ 
sufficiently large. The technical issue is to determine lower bounds for the parameter $c$, 
which guarantee these implications. 

The $L_c$-stability of a sheaf is an open property for $c>0$, independent of the relative 
semi-stability. One typically obtains different (semi-)stability conditions \eqref{eq:ss}, 
as the parameter $c\ges 0$ varies. The effect of the relative semi-stability is that of 
stabilizing the various concepts. 

\begin{lemma}\label{lm:eps}
{\rm(i)} 
The set $I:=\{c\in\mbb R_{>0}\mid \cF\text{ is }\LL_c\text{-(semi-)stable}\}$ 
is an interval. 

\nit{\rm(ii)} 
Assume that $\cF$ is $\LL_c$-(semi-)stable, and relatively semi-stable. 
Then $\cF$ is $\LL_{c+\veps}$-(semi-)stable, for all $\veps>0$. 
\end{lemma}

\begin{proof}
{\rm(i)} Let $a,b\in I$, and $a<c<b$. Then $c=(1-\l) a+\l b$ for some 
$\l\in(0,1)$, and for any subsheaf $\cG\subset\cF$ we have  
$\mu_{\LL_c}(\cG)=(1-\l)\cdot \mu_{\LL_a}(\cG)+\l\cdot \mu_{\LL_b}(\cG)$. 

\nit{(ii)} Indeed, one has 
$\mu_{L_{c+\veps}}(\cG)=\mu_{L_{c}}(\cG)+\veps\mu_{L,\fbr}(\cG_\fbr)$ 
for any saturated subsheaf $\cG\subset\cF$. 
\end{proof}

Finally, let us remark that the definition \eqref{eq:Lc-ss} of the slope 
differs from the usual one 
\begin{equation}\label{eq:binom}
\begin{array}{ll}
\mu^{\rm usual}_{L_c}(\cG)
&
:=
\xi_\cG L_c^{d_Y-1}
\\ 
&=
\xi_\cG\Big[
L^{d_Y-1}+\ldots
+c^{d_X-1}\binom{d_Y-1}{d_X-1}A^{d_X-1}L^{d_Y-d_X}
+c^{d_X}\binom{d_Y-1}{d_X}A^{d_X}L^{d_Y-d_X-1}
\Big].
\end{array}
\end{equation}
By using our result, we can compare the two (semi-)\-stability concepts. The outcome is 
analogous to the relationship between the Gieseker and the (usual) slope (semi-)stability.

\begin{proposition}\label{prop:compare} 
$\begin{array}[t]{lcll}
\LL_c\text{-stable }\eqref{eq:Lc-ss}
&\Rightarrow&
\text{usual }\LL_c\text{-stable }\eqref{eq:binom},&\text{ for }c\gg 0;
\\ 
\text{usual }\LL_c\text{-semi-stable }\eqref{eq:binom}
&\Rightarrow&
\LL_c\text{-semi-stable }\eqref{eq:Lc-ss},&\text{ for }c\gg0.
\end{array}$\\ 
Consequently, the main theorem still holds for (usually) $\LL_c$-(semi-)stable sheaves. 
\end{proposition}

\begin{proof} 
View \eqref{eq:binom} as a polynomial in the indeterminate $c$, and 
observe that the two (rightmost) terms are, up to a scaling factor, precisely the 
slope \eqref{eq:Lc-ss}. Our main result provides the bounds (depending on the 
numerical data of $\cF$ only), necessary for proving the two implications. 
\end{proof}

\nit If one is interested in the usual $\LL_c$-slope (semi-)stability, is still possible 
to deduce effective bounds in this setting, albeit more involved. 
Below are a couple of examples. 
\begin{enumerate}
\item For $d_X=1$, that is $X$ is a curve, holds 
$\mu_{L_c}^{\rm usual}(\cG)=\mu_{L_{(d_Y-1)c}}(\cG)$, 
for any sheaf $\cG$ on $Y$. Thus the constant $k_\cF$ in the introduction 
gets replaced by $k'_\cF:=k_\cF/(d_Y-1)$. 

\item For $d_X=2$, that is $X$ is a surface, holds 
$L_cAL^{d_Y-3}=AL^{d_Y-2}+cA^2L^{d_Y-3}$, and 
$\LL_c^{d_Y-1}=c(d_Y-1)\Big[
\frac{1}{c(d_Y-1)}L^{d_Y-1}+AL^{d_Y-2}+c\frac{d_Y-2}{2}A^2L^{d_Y-3}
\Big].$

\nit If $s{:=}\,\mu_L^{\rm usual}(\cG)-\mu_L^{\rm usual}(\cF)$, where $\cG$ is the 
first term of the (usual) Harder-Narasimhan filtration of $\cF$ with respect to $L$, 
then the main theorem holds for  
$k'_\cF{:=}\max\!\Big\{{\frac{2k_\cF}{d_Y-2},\frac{r^2\cdot s}{d_Y-1}}\Big\}$.
\end{enumerate}


\subsection{Relative semi-stability in terms of the slope of $\cF$}\label{ssct:slope}

The following lemma is inspired from \cite[pp. 263]{la}. 

\begin{lemma}{\label{lm:infty}}
For $c$ sufficiently large, the first term of the Harder-Narasimhan filtration 
of $\cF$ with respect to $\LL_c$ is independent of $c$. More precisely, it holds 
\begin{equation}{\label{eqn:c0F}}
\hn{F}{c}=\hnr{F},\quad\forall\,c>a_\cF:=r^2(M_\cF-m_\cF)/A^{d_X},
\end{equation}
with $M_\cF:=\mu_L(\hn{F}{})$ and $m_\cF:=\mu_L(\hnr{F})$. 
In particular, if $\cF$ is $\LL_a$-(semi-)stable for some $a>a_\cF$, 
then $\cF$ is $\LL_c$-(semi-)stable for all $c\ges a$. 
\end{lemma}

\begin{proof}
The slope of a subsheaf $\cG\subset\cF$ with respect to $\LL_c$ is
$\mu_{L_c}(\cG)=c\cdot\mu_{L,\,\fbr}(\cG)+\mu_L(\cG)$. We endow the set 
$\mbb S(\cF)$ of all polynomials in $c$ of the form $\mu_{L_c}(\cG)$ above, 
corresponding to some $\cG\subset\cF$, with the lexicographic order 
(for which the indeterminate $c$ is greater than $1$). 
The coefficients of the polynomials in $\mbb S(\cF)$ are bounded from above 
by $\mu_{L,\,\fbr}(\hnr{F})$ and $\mu_L(\hn{F}{})$ respectively, so 
$\mbb S(\cF)$ admits a maximal element $\mbb S(\cF)_\mx$. Let us determine it. 
The coefficient of $c$ is at most $\mu_{L,\,\fbr}(\hnr{F})$, and is attained 
for the subsheaves $\cG\subset\cF$ such that 
$\mu_{L,\fbr}(\cG)=\mu_{L,\fbr}(\hnr{F})$. 
Then the maximal polynomial is: 
$$
\begin{array}{ll}
\mbb S(\cF)_\mx\kern-4pt
&=
c\cdot\mu_{L,\,\fbr}(\hnr{F})
+
\max\left\{\mu_L(\cG)\left| 
\begin{array}{l}
\cG\subset\cF\text{ subsheaf,}
\\ 
\mu_{L,\,\fbr}(\hnr{G})=\mu_{L,\,\fbr}(\hnr{F})
\end{array}
\right.\kern-5pt
\right\}.
\end{array}
$$
\nit\textit{Claim}\quad The maximum above equals $\mu_L(\hnr{F})$. 
Indeed, let $\cG$ be such that 
$\mu_{L,\,\fbr}(\hnr{G})$ $=\mu_{L,\,\fbr}(\hnr{F})$, and $\mu_L(\cG)$ 
is maximal with this property. The $L$-slope of $\cG$ increases by taking its saturation 
(as $L$ is semi-ample and big), so we may assume that $\cG\subset\cF$ is saturated. 
The uniqueness of the Harder-Narasimhan filtration of $\cF_\fbr$ implies 
$\hnr{G}_\fbr=\hnr{F}_\fbr$, so $\cG\subset\hnr{F}$ by the maximality 
of $\hnr{F}$ (see \ref{eq:[x]}). Hence $\hnr{F}/\cG$ is a torsion sheaf 
which vanishes over the generic fibre. Its support is a proper subscheme 
$Z\subset Y$ such that $\pi_*Z\subset X$ is also proper. It follows that 
$$
\begin{array}{c}
\mu_L(\hnr{F})=\mu_L(\cG)+
\underset{\scalebox{.6}{$\begin{array}{c} Z'\subset Z\\[.5ex]
\text{irred. divisor on }Y\end{array}$}}{\sum}
\rho_{Z'}\cdot Z'A^{d_X-1}L^{d_Y-d_X},\quad\text{with all }\rho_{Z'}>0.
\\[-1.5ex]
\end{array}
$$
But $\codim_X\pi_*Z'=1$ because $\pi_*Z\subset X$ is proper, hence 
$Z'A^{d_X-1}L^{d_Y-d_X}$ is strictly positive. 

Overall, we proved that 
$\mbb S(\cF)_\mx=\mu_{L,\fbr}(\hnr{F})c+\mu_{L}(\hnr{F})$. To complete 
the proof, is enough to show that if 
$\mu_{L_c}(\cG)<_{\rm lex}\mbb S(\cF)_\mx$ (as a polynomials in $c$) 
for a subsheaf $\cG\subset\cF$, 
then $\mu_{L_c}(\cG)<\mbb S(\cF)_\mx(c)$ for all $c>a_\cF.$ 
There are two possibilities:
\smallskip 

\nit\textit{Case 1}\quad $\mu_{L,\fbr}(\hnr{F})>\mu_{L,\,\fbr}(\cG)$ 
$\Rightarrow$ 
$\mu_{L,\fbr}(\hnr{F})-\mu_{L,\,\fbr}(\cG)\ges\frac{A^{d_X}}{r^2}$. 
Then follows: 
$$
\begin{array}{ll}
\mbb S(\cF)_\mx(c)-\mu_{L_c}(\cG)
&
=c\bigl(\mu_{L,\fbr}(\hnr{F})-\mu_{L,\,\fbr}(\cG)\bigr)
+
\bigl(\mu_{L}(\hnr{F})-\mu_L(\cG)\bigr)
\\[1ex]
&
\ges
\frac{cA^{d_X}}{r^2}+m_\cF-M_\cF
=
\frac{1}{r^2}\!\cdot\! 
\bigl[
cA^{d_X}-r^2(M_{\cF}-m_\cF)
\bigr]
>0.
\end{array}
$$

\nit\textit{Case 2}\quad $\mu_{L,\fbr}(\hnr{F})\!=\!\mu_{L,\,\fbr}(\cG)$,   
$\mu_{L}(\hnr{F})\!>\!\mu_L(\cG)$. 
Then holds $\mbb S(\cF)_\mx(c)-\mu_{L_c}(\cG)\!>\!0$. 

\nit The last statement is a direct consequence of lemma \ref{lm:eps}.  
\end{proof}

\begin{theorem}{\label{thm:main0}} 
Let $a_\cF$ be as in \eqref{eqn:c0F}. Then the following hold: 
\begin{enumerate}
\item 
If $\cF$ is $\LL_a$-semi-stable with $a>a_\cF$, then $\cF_\fbr$ is semi-stable.

\item  
If $\cF_\fbr$ is stable, then $\cF$ is $\LL_c$-stable for all $c>a_\cF$. 
\end{enumerate}
\end{theorem}

\begin{proof}
(i) The previous lemma implies that $\cF$ is $\LL_c$-semi-stable for all 
$c\ges a$, and therefore $\hnr{F}=\hn{F}{c}=\cF$. Hence $\cF_\fbr$ is 
indeed semi-stable.\smallskip  

\nit(ii) Conversely, assume that $\cF_\fbr$ is stable, so $\hnr{F}=\cF$ and 
$m_\cF=\mu_L(\cF)$. If there is a destabilizing proper, saturated subsheaf 
$\cG$ of $\cF$, then 
$$
\begin{array}{rl}
&\mu_L(\cF)+c\cdot\mu_{L,\fbr}(\cF)=\mu_{L_c}(\cF)
\les
\mu_{\LL_c}(\cG)
=\underbrace{\mu_{L}(\cG)}_{\les\,M_\cF}+\, c\cdot\kern-1ex 
\underbrace{\mu_{L,\fbr}(\cG)}_{\les\,\mu_{L,\fbr}(\cF)-A^{d_X}\!/r(r-1)}
\\[1ex]
\Rightarrow&
m_\cF\les M_\cF-cA^{d_X}\!/r(r-1).
\end{array}
$$
This contradicts the choice of $c>r^2(M_\cF-m_\cF)/A^{d_X}$, 
so $\cF$ is $\LL_c$-stable. 
\end{proof}


\subsection{Relative semi-stability in terms of the Chern classes of $\cF$}
\label{ssct:char}

Here we derive a result analogous to theorem \ref{thm:main0} above, with the difference 
that the lower bound for the parameter $c$ is expressed in terms of the characteristic 
classes of $\cF$. For a torsion free sheaf $\cG$ on $Y$, we denote 
$\Delta(\cG)=2\rk(\cG)c_2(\cG)-(\rk(\cG)-1)c_1^2(\cG)\in H^4(Y;\mbb Q)$ 
the {\em discriminant} of $\cG$. For shorthand, let 
$$
\al{\gamma}:=\gamma\cdot A^{d_X-1}L^{d_Y-d_X-1}, 
\text{ for all }\gamma\in H^4(Y;\mbb Q).
$$ 
We consider the `light cone' 
$$
K^+:=\{
\beta\in \NS(Y)_{\mbb Q}\mid 
\al{\beta^2}>0\text{ and }\al{\beta\cdot D}\ges 0,
\text{ for all nef divisors }D\subset Y
\},
$$
and we define 
\begin{equation}\label{eq:C}
C(\alpha):=\{\beta\in\ovl{K^+}\mid\al{\alpha\cdot\beta}>0\}, 
\quad\forall\;\alpha\in\NS(Y)_{\mbb Q}\sm\{0\}.
\end{equation}

\begin{proposition}\label{prop:c0} 
{\rm(i)} Let $\cF$ be a torsion free sheaf on $Y$ with $c_1(\cF)=0$. 

\nit{$\rm(i_a)$} If $\cF$ is not $\pi$-relatively semi-stable, then there is a proper 
saturated subsheaf $\cG$ of $\cF$ such that: 
\begin{eqnarray}
\mbox{$\mu_{L,\Phi}({\cG})\ges\frac{A^{d_X}}{r-1}$}, &
\label{eq:mu1}
\\ 
\text{and either}\quad
\al{\xi_{\cG}^2}\ges 0,&
\text{or}\quad
0>\al{\xi_{\cG}^2}\ges-\frac{2r}{r-1}\cdot \al{c_2(F)}.
\label{eq:mu2}
\end{eqnarray}

\nit{$\rm(i_b)$} If $\cF$ is not $L_c$-stable, then there is a proper 
saturated subsheaf $\cG\subset\cF$ such that $\mu_{L_c}(\cG)\ges0$, and $\xi_\cG$ 
satisfies one of \eqref{eq:mu2}.

\nit{\rm(ii)} 
The statements $\rm(i_a)$ (respectively $\rm(i_b)$) still hold for $c_1(\cF)$ arbitrary, 
with the modifications: 
\begin{eqnarray*}
\mbox{$\mu_{L,\fbr}(\cG)\ges\mu_{L,\fbr}(\cF)+\frac{A^{d_X}}{r(r-1)}$}, 
&
(\text{respectively},\;\mu_{L_c}(\cG)\ges\mu_{L_c}(\cF)\,),
&\kern3em\eqref{eq:mu1}'
\\ 
\text{and either}\quad
\al{(\xi_{\cG}-\xi_\cF)^2}\ges 0,
& 
\text{or}\quad
0>\al{(\xi_{\cG}-\xi_\cF)^2}\ges-\frac{\Delta(\cF)}{r-1}.
&\kern3em\eqref{eq:mu2}'
\end{eqnarray*}
\end{proposition}

The result is similar to the Bogomolov inequality \cite[Theorem 7.3.4]{hl} and 
\cite[Theorem 3.12]{la}, with the difference that here we \emph{simultaneously} 
control the \emph{discriminant} and the \emph{relative slope} of the relatively 
de-semi-stabilizing subsheaf. 

\begin{proof}
$\rm(i_a)$ 
As $\cF_\fbr$ is not semi-stable, $\mu_{L,\Phi}(\hnr{F})\ges\frac{A^{d_X}}{r-1}>0$. 
For shorthand we write $\cF':=\hnr{F}$ and let 
$\cF'':=\cF/\cF'$. If $\al{\xi_{\cF'}^2}\ges 0$, there is nothing to prove, 
so let us assume $\al{\xi_{\cF'}^2}< 0$. (Thus, in particular, 
$\xi_{\cF'}\not\in-\ovl{K^+}$ and $C(\xi_{\cF'})\neq\emptyset$.) \smallskip 

\nit\textit{Case 1}\quad Assume that holds 
$\al{\Delta(\cF')},\al{\Delta(\cF'')}\ges 0$. 
The equality (see \cite[pp. 207]{hl}) 
$$\begin{array}{c}
\frac{{[\Delta(\cF)]}_{AL}+r'r''{[\xi_{\cF'}^2]}_{AL}}{r}
=
\frac{{[\Delta(\cF')]}_{AL}}{r'}+\frac{{[\Delta(\cF'')]}_{AL}}{r''}
\end{array}
$$
implies that $\al{\Delta(\cF)}>0$ and 
$\al{\xi_{\cF'}^2}\ges-\frac{{[\Delta(\cF)]}_{AL}}{r'r''}\ges
-\frac{{[\Delta(\cF)]}_{AL}}{r-1}=-\frac{2r}{r-1}\cdot\al{c_2(\cF)}$. 
\smallskip

\nit\textit{Case 2}\quad Assume that holds $\al{\Delta(\cF')}<0$. According to  
\cite[Theorem 3.12]{la}, there is a saturated subsheaf $\cG'\subset\cF'$ 
such that $\xi_{\cG'}-\xi_{\cF'}\in K^+$, so 
$\al{(\xi_{\cG'}-\xi_{\cF'})\cdot A}\ges0$. 
As $\xi_{\cG'}=\bigl(\xi_{\cG'}-\xi_{\cF'}\bigr)+\xi_{\cF'}$, we deduce   
$\mu_{L,\Phi}(\cG')>0$, hence $\mu_{L,\Phi}(\cG')\ges\frac{A^{d_X}}{r-1}$, 
and $C(\xi_{\cG'})\supsetneq C(\xi_{\cF'})$. 
\smallskip

\nit\textit{Case 3}\quad Assume that holds $\al{\Delta(\cF'')}<0$. 
As before, there is a saturated subsheaf $\cG''\subset\cF''$ 
such that $\xi_{\cG''}-\xi_{\cF''}\in K^+$. 
For $\cG:=\Ker(\cF\to\cF''/\cG'')$ holds  
$\xi_{\cG}=\rho'\cdot\xi_{\cF'}+\rho''\cdot(\xi_{\cG''}-\xi_{\cF''})$, 
with $\rho',\rho''>0$ (see \cite[pp. 206, equation (7.6)]{hl}). 
Once again, this implies  
$\mu_{L,\Phi}({\cG})>0$, hence $\mu_{L,\Phi}({\cG})\ges\frac{A^{d_X}}{r-1}$, 
and also $C(\xi_{\cG})\supsetneq C(\xi_{\cF'})$.

In both cases 2 and 3 we can replace $\cF'$ with another saturated subsheaf of $\cF$ 
which is still relatively de-semi-stabilizing (but not necessarily of maximal slope), 
and the corresponding cone \eqref{eq:C} is strictly larger. But this increasing sequence 
of cones stops because there are only finitely many possibilities for them. 
(See the proof of \cite[Theorem 7.3.3]{hl}.) Thus, after finitely many steps, we reach 
either the case 1, or the case $\al{\xi^2}\ges0$. The proof of $\rm(i_b)$ is identical.
\smallskip 

\nit$\rm(ii)$ The proof is similar, except that one has to replace overall $\xi_\cG$ by 
$\xi_\cG-\xi_\cF$. (This is the explanation for the weaker bound \eqref{eq:mu1}$'$.)
\end{proof}

Now we derive an inequality which relates the fibrewise and the absolute slope 
of a saturated sheaf on $Y$. The equality 
\begin{equation}\label{eq:equal}
\al{\;\bigl(\,\al{\xi\cdot A} L_c-\al{A\cdot L_c}\xi\,\bigr)\cdot A\;}=0
\end{equation} 
holds for any $\xi\in\NS(Y)_{\mbb Q}$ and $c\ges 0$. As $L$ on $Y$ is semi-ample 
and $A$ on $X$ is ample, we can view this expression as the intersection product on a 
smooth (complete intersection) surface in $Y$, representing (a multiple of) the class 
$A^{d_X-1}L^{d_Y-d_X-1}$, so the Hodge index theorem yields: 
\begin{equation}\label{eq:inequal}
\begin{array}{rl}
& 
0\ges\al{\;\bigl(\,\al{\xi\cdot A}L_c-\al{A\cdot L_c}\xi\,\bigr)^2\;}
\\[2mm] 
\Rightarrow\,&
2\kern-1ex\underbrace{\al{A\cdot L_c}}_{=A^{d_X}L^{d_Y-d_X}>0}\kern-2ex
\cdot\,\al{\xi\cdot A}\cdot\al{\xi\cdot L_c}
\ges
\al{\xi\cdot A}^2\cdot\underbrace{\al{L_c^2}}_{(*)}
+{\bigl(A^{d_X}L^{d_Y-d_X}\bigr)}^2\kern-.5ex\cdot\al{\xi^2},
\end{array}
\end{equation}
and the marked term above is 
$$
(*)=
A^{d_X-1}L^{d_Y-d_X-1}L_c^2
=A^{d_X-1}L^{d_Y-d_X+1}+2c\cdot A^{d_X}L^{d_Y-d_X}
\ges2c\cdot A^{d_X}L^{d_Y-d_X}.
$$
After dividing both sides of \eqref{eq:inequal} by $A^{d_X}L^{d_Y-d_X}$, we deduce 
\begin{equation}\label{eq:ineq}
\begin{array}{c}
2\al{\xi\cdot A}\cdot\al{\xi\cdot L_c}
\ges 
2c\cdot\al{\xi\cdot A}^2+A^{d_X}L^{d_Y-d_X}\cdot \al{\xi^2}.
\end{array}
\end{equation}

\begin{theorem}{\label{thm:main}}
{\rm(i)} Assume $c_1(\cF)=0$, and let   
$c_\cF\,{:=}\,
r(r-1)\cdot\frac{A^{d_X}L^{d_Y-d_X}}{{(A^{d_X})}^2}\cdot\al{c_2(\cF)}$.  
The following statements hold: 
\begin{enumerate}
\item[$\rm(i_a)$] 
If $\cF$ is $\LL_a$-semi-stable with $a>c_\cF$, then $\cF_\fbr$ is semi-stable. 
In particular, if $\cF$ is $L_a$-semi-stable, then it is $L_c$-semi-stable for all $c\ges a$. 
\item[$\rm(i_b)$] 
If $\cF_\fbr$ is stable, then $\cF$ is $\LL_c$-stable for all $c>c_\cF$. 
\end{enumerate}
\nit{\rm(ii)} 
For $c_1(\cF)$ arbitrary,  $\rm(i_a)$, $\rm(i_b)$ still hold for 
$c'_\cF\,{:=}\,\frac{r^2(r-1)}{2}\cdot\frac{A^{d_X}L^{d_Y-d_X}}{{(A^{d_X})}^2}
\cdot\al{\Delta(\cF)}$.
\end{theorem}

\begin{proof}
(i) Indeed, assume that $\cF_\Phi$ is not semi-stable. Then there is a subsheaf $\cG$ of 
$\cF$ satisfying proposition \ref{prop:c0}, 
so $\mu_{L,\Phi}(\cG)\ges\frac{A^{d_X}}{r-1}$. 
By replacing $\xi=\xi_\cG$ in \eqref{eq:ineq}, we obtain 
\\ \centerline{$
2\mu_{L,\Phi}(\cG)\cdot\mu_{L_a}(\cG)
\ges 
\frac{2a(A^{d_X})^2}{(r-1)^2}+A^{d_X}L^{d_Y-d_X}\cdot \al{\xi_\cG^2},$}
and the right hand side is strictly positive: for $\al{\xi_\cG^2}\ges 0$ is clear, 
and otherwise $0>\al{\xi_\cG^2}>-\frac{2r}{r-1}\al{c_2(\cF)}$. 
Thus $\mu_{L_a}(\cG)>0$, which contradicts the $L_a$-semi-stability of $\cF$. 

Conversely, if $\cF_\Phi$ is stable and $\cF$ is not $L_c$-stable over $Y$, there is 
a saturated subsheaf $\cG$ of $\cF$ such that $\mu_{L_c}(\cG)\ges 0$ and 
$\mu_{L,\Phi}(\cG)<0$. As before, the right hand side of the previous inequality 
is strictly positive, so $\mu_{L_c}(\cG)<0$, a contradiction. 

\nit(ii) Repeat the argument by using proposition \ref{prop:c0}(ii). 
\end{proof}

\begin{remark}\label{rmk:weak}
Recall that we required $L$ to be big, $\pi$-ample and semi-ample. Is possible to slightly 
weaken the bigness assumption, that is $L^{d_Y}>0$. Proposition \ref{prop:c0}, hence 
also theorem \ref{thm:main}, still hold for $L^{d_Y}=0$ and $AL^{d_Y-1}>0$. 
Indeed, in this case, both equations \eqref{eq:equal} and \eqref{eq:inequal} hold for $L$ 
replaced by $L_\veps$, with $\veps>0$ small, and \eqref{eq:ineq} follows by a limiting 
argument.
\end{remark}


\subsection{Relative semi-stability for principal bundles}{\label{sect:pb}}

Our previous conclusions can be generalized to principal bundles with reductive 
structure groups. First we introduce the semi-stability concept with respect to 
$L_c,L,A$ (we call it $L_c$-(semi-)stability), analogous to \eqref{eq:Lc-ss}.

\begin{definition}
(i) A principal $G$-bundle $\O$ on $Y$, where $G$ is a connected reductive linear group, 
is $L_c$-(semi-)stable if for any parabolic subgroup $P\subset G$ and any reduction 
$\si:\cal U\to \O_{\cal U}/P$ defined over an open subset $\cal U\subset Y$ whose 
complement has co-dimension at least two in $Y$ holds 
\begin{equation}\label{eq:Lc-ss-pb}
\deg_{L_c}(\si^*T^\rel_{\O_{\cal U}/P}):=
c_1(\si^*T^\rel_{\O_{\cal U}/P})\cdot L_c A^{d_X-1}L^{d_Y-d_X-1}
\underset{(\ges)}{>}0,
\end{equation}
where $T^\rel_{\O_{\cal U}/P}$ stands for the relative tangent bundle on $\O_{\cal U}/P$. 

\nit(ii) The semi-stability of the restriction $\O_\fbr$ is defined with respect to $L_\fbr$, 
as usual. 
\end{definition}

\begin{theorem}\label{thm:pb}
Let $G$ be a connected, reductive algebraic group, and $\O$ be a principal $G$-bundle 
on $Y$. There is a constant $c_\O$, such that the following hold: 
\begin{enumerate}
\item If $\O$ is $\LL_a$-semi-stable, with $a>c_\O$, then its restriction 
$\O_\fbr$ is $L$-semi-stable. 
\item If $\O_\fbr$ is $L$-stable, then $\O$ is $\LL_c$-stable, for all $c>c_\O$. 
\end{enumerate}
\end{theorem}

\begin{proof} 
(i) The principal bundle $\O$ is semi-stable if and only if the vector bundle $\ad(\O)$, 
induced by the adjoint representation $G\to\Gl\bigl(\End(\mfrak g)\bigr)$ of $G$, 
is semi-stable (see \cite[Corollary 3.18]{ra}). 

\nit(ii) Let us assume that $\O$ is not $L_c$-stable. Then there is an open subset 
$\cal U\subset Y$, whose complement has co-dimension at least two in $Y$, 
and a reduction $(P,\si)$ of $\O$ over $\cal U$, such that 
\begin{equation}\label{eq:U}
\deg_L(\si^*T^\rel_{\O_{\cal U}/P})
+
c\cdot c_1(\si^*T^\rel_{\O_{\cal U}/P})A^{d_X}L^{d_Y-d_X-1}
=
\deg_{L_c}(\si^*T^\rel_{\O_{\cal U}/P})\les 0.
\end{equation}
The restriction of $(P,\si)$ to the generic fibre still defines a reduction of $\O_\fbr$ 
over $\cal U\cap\fbr$, and the complement of this latter in $\fbr$ has co-dimension 
at least two, too. The stability of $\O_\fbr$ implies 
$$
c_1(\si^*T^\rel_{\O_{\cal U}/P})A^{d_X}L^{d_Y-d_X-1}
=
A^{d_X}\deg_{L_\fbr}(\si^*T^\rel_{\O_\fbr/P})\ges A^{d_X}>0.
$$ 
By inserting this into \eqref{eq:U}, we obtain 
$c\les-\frac{1}{A^{d_X}}\deg_L(\si^*T^\rel_{\O_{\cal U}/P}).$ 
But, for any reduction $(P,\si)$, the vector bundle $\si^*T^\rel_{\O_{\cal U}/P}$ is a 
quotient of $\ad(\O_{\cal U})$, and this latter extends to a torsion free quotient of 
$\ad(\O)$. (The slope of the quotient is preserved in this process, as 
$\codim_Y(Y\sm\cal U)\ges2$.) Let $k_{\ad(\O)}$ be a lower bound for the 
$L$-slopes of the quotients of $\ad(\O)$. Then the previous equation implies 
$c\les-\frac{\rk(\ad(\O))\cdot k_{\ad(\O)}}{A^{d_X}},$ which contradicts that 
$c$ is sufficiently large. 
\end{proof}


\section{Application: stable vector bundles on Hirzebruch surfaces}{\label{sect:hirz}}

For $\ell\ges 0$, the Hirzebruch surface 
$Y_\ell:=\mbb P(\cal O_{\mbb P^1}\oplus\cal O_{\mbb P_1}(-\ell))$ is a 
$\mbb P^1$-fibre bundle over $\mbb P^1$. Let $\pi\,{:}\,Y_\ell\,{\rar}\,\mbb P^1$ be the 
projection, and $\cal O_\pi(1)$ the relatively ample line bundle. The `exceptional line' 
$\Lambda\,{=}\,\mbb P(\cal O_{\mbb P^1}{\oplus}\,0)\,{\hra}\,Y_\ell$ 
has self-intersection $\Lambda^2{=}\,-\ell$, 
$\cal O_\pi(1)\,{=}\,\ell\cal O_{\mbb P^1}(1){+}\,\cO_Y(\Lambda)$, 
and the relative and (absolute) canonical classes of $Y_\ell$ are respectively 
\begin{equation}\label{eq:K}
\kappa_{Y_\ell/\mbb P^1}=\cal O_\pi(-2)+\ell\pi^*\cal O_{\mbb P^1}(1)
\text{ and }
\kappa_{Y_\ell}=\cal O_\pi(-2)+(\ell-2)\pi^*\cal O_{\mbb P^1}(1).
\end{equation}
The goal of this section is to study the geometry of the moduli space 
$\bar M_{Y_\ell}^{L_c}(r;0,n)$ of torsion free sheaves on $Y_\ell$, of rank $r$, 
with $c_1(\cF)=0$ and $c_2(\cF)=n$, which are semi-stable with respect to 
$\LL_c=\cal O_\pi(1)+c\cdot\pi^*\cal O_{\mbb P^1}(1)$. 
Our approach is similar to \cite[Section 1]{gl}, although in \emph{loc. cit.} the authors 
consider polarizations $L_c$ with $0<c\ll1$ (while we consider $c\gg0$). 

\subsection{Construction of semi-stable sheaves on $Y_\ell$}

\begin{lemma}\label{lm:1}
Let $\cF$ be an $L_c$-semi-stable torsion free sheaf of rank $r$ on $Y_\ell$, 
with $c>r(r-1)n$, $c_1(\cF)=0$ and $c_2(\cF)=n$. 
Then the following statements hold: 
\begin{enumerate}
\item We have $\fbr\cong\mbb P^1$ and $\cF_\fbr\cong\cal O_{\fbr}^{\oplus r}$. 
If $\cF$ is $L_c$-stable, then $h^1(Y_\ell,\cF)=n-r$, so $n\ges r$. 

\item 
The Chern character of the derived direct image $\pi_!\cF=\pi_*\cF-R^1\pi_*\cF$ is 
\begin{equation}\label{eq:piF}
\ch(\pi_!\cF)=\ch_0(\pi_!\cF)\oplus\ch_1(\pi_!\cF)=r\oplus -n.
\end{equation}

\item 
The natural homomorphism $f_\cF\;{:}\;\pi^*\pi_*\cF{\to}\cF$ is injective, 
and $\det(f_\cF)^\vee\in|\pi^*\cO_{\mbb P^1}(n_\cF)|$ with $n_\cF\les n$. We denote 
by $Z_\cF{=}\underset{i\in I}{\sum}m_ix_i$ the divisor of $\det(f_\cF)^\vee$, 
with $x_i\in\mbb P^1$, $m_i\ges 1$, and $\underset{i\in I}{\sum}m_i=n_\cF$. 
The sheaf $R^1\pi_*\cF$ is torsion on $\mbb P^1$, 
and $\deg_{\mbb P^1}(R^1\pi_*\cF)=n-n_\cF$. 

\item 
$\pi_*\cF$ is locally free of rank $r$ on $\mbb P^1$, so it splits:  
\begin{equation}\label{eq:a}
\pi_*\cF\cong\ouset{j=1}{p}{\mbox{$\bigoplus$}}\cO(-a_j)^{\oplus r_j}, 
\text{ with }
0\les a_1<\ldots<a_p
\quad\text{and }\left\{
\begin{array}{l}
r_1+\ldots+r_p=r,\\ 
a_1r_1+\ldots+a_pr_p=n_\cF.
\end{array}
\right.
\end{equation}
If $\Gamma(Y_\ell,\cF){=}0$, then $a_j\ges 1$ for all $j$. 

\item 
$\pi_*\cF(-\Lambda)=0$ and $R^1\pi_*\cF(-\Lambda)$ is a torsion sheaf on $\mbb P^1$, 
with  $\deg R^1\pi_*\cF(-\Lambda)=n$.
\end{enumerate}
\end{lemma}

\begin{proof}
\nit(i) By theorem \ref{thm:main}, $\cF_\fbr$ is semi-stable, so 
$\cF_\fbr\cong \cal O_{\mbb P^1}^{\oplus r}$ because $\fbr\cong\mbb P^1$. 
(For $\ell=0$, see remark \ref{rmk:weak}.) If $\cF$ is stable, then 
$\Gamma(Y,\cF)=H^2(Y,\cF)=0$, and $\dim H^1(\cF)=n-r$ 
by the Riemann-Roch formula.  
\smallskip

\nit(ii)-(v) If $a_1<0$, then $\cO_{\mbb P^1}\subset\cF(a_1)$, which contradicts the 
semi-stability of $\cF$. Further, as $\cF(-\Lambda)_\fbr\cong\cO_\fbr(-1)^{\oplus r}$, 
we deduce $\Gamma(U,\pi_*\cF(-\Lambda))=0$, for all open $U\subset\mbb P^1$. 
The Grothendieck-Riemann-Roch theorem yields the formula for ${\rm ch}(\pi_!\cF)$.
\end{proof}

\begin{lemma}\label{lm:QQ}
The sheaf $\cal Q_\cF$ defined by 
$0\,{\to}\,\pi^*\pi_*\cF\,{\srel{f_\cF}{\to}}\,\cF{\to}\,\cal Q_\cF\,{\to}\,0$ 
has the following properties: 
\begin{enumerate}
\item Any local section $\si$ through a closed point $p\in\Supp(\cal Q_\cF)$ is 
$\cal Q_{\cF,p}$-regular, that is the multiplication 
$\cal Q_{\cF,p}\to\cal Q_{\cF,p}$ by the defining equation of $\si$ is injective. 
Thus the homological dimension and the depth of $\cal Q_{\cF,p}$ are both equal one. 

\item $\Supp(\cal Q_\cF)\,{=}\,\pi^{-1}(Z_\cF)$, and its Chern classes are 
$c_1(\cal Q_\cF)=n_\cF\cdot[\cO_{\mbb P^1}(1)],\;c_2(\cal Q_\cF)=n$. 

\item There is an exact sequence 
$0{\to}\,\cal Q_\cF\,{\otimes}\,\cO_\Lambda{\to}\,
R^1\pi_*\cF(-\Lambda)\,{\to}\,R^1\pi_*\cal Q_\cF{\to}\,0.$ 

\item If $\det(f_\cF)^\vee\!$ has simple zeros $x_1,\ldots,x_{n_\cF}\in\mbb P^1$, 
then $\cal Q_\cF\,{=}\,\ouset{i=1}{n_\cF}{\bigoplus}\cO_{\pi^{-1}(x_i)}(-b_i)$, 
with $b_i\ges 1$ and $b_1+\ldots+b_{n_\cF}=n$. 
(Thus, for $n_\cF=n$, we have $b_1=\ldots=b_n=1$.)

\item An isomorphism $\cF\srel{\th}{\to}\cF'$ induces the commutative diagram: 
$$
\xymatrix@R=1.5em{
0\ar[r]&
\pi^*\pi_*\cF\ar[r]\ar[d]_-{\pi_*\th}^-\cong&
\cF\ar[r]\ar[d]_-\th^-\cong&
\cal Q_\cF\ar[r]\ar[d]_-{\hat\th}^-\cong&0
\\ 
0\ar[r]&
\pi^*\pi_*\cF'\ar[r]&
\cF'\ar[r]&
\cal Q_{\cF'}\ar[r]&0.
}
$$
\end{enumerate}
\end{lemma}

\begin{proof}
\nit(i) We choose local coordinates $x,y$ around $p$, with $p=(0,0)$ and $\pi$ 
given by $(x,y)\mt x$, and $\si(x)=(x,0)$. (So, the local equation defining $\si$ is $y$.) 
As $\Supp(\cal Q_\cF)=\pi^{-1}(Z_\cF)$, there is $m>0$ such that $\cal Q_{\cF,p}$ 
is annihilated by $\langle x^m\rangle$. 
Assume that the multiplication by $y$ is not injective. 
Then there is a non-trivial, zero-dimensional submodule of $\cal Q_{\cF,p}$ annihilated 
by $\cal I:=\langle x^m,y\rangle\subset\cal O_{Y_\ell,p}$, so $\pi^*\mbb L\subset\cF$ 
is non-saturated. The saturation $\cal G\subset\cF$ of $\pi^*\mbb L$ has the property that 
$\cal Q':=\cal G/\pi^*\mbb L\subset\cal Q_\cF$ is non-trivial, zero-dimensional, and \\ 
\centerline{$
\cal G\in\Ext^1(\cal Q',\pi^*\mbb L)\ouset{\text{loc. free}}{\pi^*\mbb L}{\cong}
\Ext^1(\pi^*\mbb L,\kappa_Y\otimes\cal Q')^\vee=0.
$}
It follows $\cal G=\cal Q'\oplus\pi^*\mbb L\subset\cF$, which contradicts that $\cF$ 
is torsion free. The second statement follows from the Auslander-Buchsbaum formula 
(see \textit{e.g.} \cite[Section 1.1]{hl}). 

\nit(ii)+(iii) The identities $\pi_*\cal Q_\cF=0, R^1\pi_*\cF\cong R^1\pi_*\cal Q_\cF$, 
are immediate, and $\Coker(\pi_*\cF\to\cF_\Lambda)=\cal Q_\cF\otimes\cO_\Lambda$ 
follows by restricting to $\L$. (Use that $\L$ is $\cal Q_\cF$-regular.) 
Now insert into $0{\to}\,\pi_*\cF{\to}\,\cF_\Lambda\to R^1\pi_*\cF(-\Lambda){\to}\,
R^1\pi_*\cF{\to}\,0$ and obtain (iii).  

\nit(iv) Indeed, $\cal Q_\cF$ is torsion free along $\pi^{-1}(x_i)$, 
so $\cal Q_\cF\,{=}\,\ouset{i=1}{n_\cF}{\bigoplus}\cO_{\pi^{-1}(x_i)}(-b_i)$. 
Then $\pi_*\cal Q_\cF\,{=}\,0$ implies that all $b_i\ges 1$,  and their sum is 
$c_2(\cal Q_\cF)=n$. 

\nit(v) The statement is obvious.
\end{proof}

\begin{lemma}\label{lm:ext}
The following equalities hold: \\ 
\centerline{$ 
\begin{array}{rl}
\ext^0(\cal Q_\cF,\pi^*\mbb L)=\ext^0(\cal Q_\cF,\cF)=0,
&\; (\cal Q_\cF\text{ is a torsion sheaf}),
\\[0ex] 
\ext^2(\cal Q_\cF,\pi^*\mbb L)=\ext^2(\cal Q_\cF,\cF)=\ext^2(\cal Q_\cF,\cal Q_\cF)=0,
&
\\[0ex] 
\ext^1(\cal Q_\cF,\pi^*\mbb L)=\ext^1(\cal Q_\cF,\cF)=r(n+n_\cF),
&\;(\text{by the Riemann-Roch formula}).
\end{array}
$}
\end{lemma}

\begin{proof}
First we notice \\ 
\centerline{$
\ext^2(\cal Q_\cF,\pi^*\mbb L)=\ext^0(\pi^*\mbb L,\kappa_{Y_\ell}\otimes\cal Q_\cF)
\srel{\Supp(\cal Q_\cF)=\pi^{-1}(Z_\cF)}{=}
=h^0(\cO_{Y_\ell}(-2\L)\otimes\cal Q_\cF)=0.
$}
The last vanishing holds because $\L$ is $\cal Q_\cF$-regular, 
so $\cO_{Y_\ell}(-2\L)\otimes\cal Q_\cF\subset\cal Q_\cF$. 
By applying the functors $\Hom(\cal Q_\cF,\cdot)$ and $\Hom(\cdot,\cF)$ respectively 
to the sequence defining $\cal Q_\cF$, we obtain: 
\begin{eqnarray}
\label{eq:s1}
0
{\to}\,
\frac{\Ext^1\!(\cal Q_\cF{,}\pi^*\mbb L)}{\End(\cal Q_\cF)}
{\to}\,
\Ext^1\!(\cal Q_\cF{,}\cF)
{\to}\,
\Ext^1\!(\cal Q_\cF{,}\cal Q_\cF)
{\to}
0
{\to}\,
\Ext^2\!(\cal Q_\cF{,}\cF){\to}\Ext^2\!(\cal Q_\cF{,}\cal Q_\cF)
{\to}\,0,
\\[0ex] 
\label{eq:s2}
\text{and}\qquad
0\,{\to}\,
\frac{\End(\mbb L)}{\End(\cF)}
\,{\to}\,
\Ext^1(\cal Q_\cF{,}\cF)
\,{\to}\,
\Ext^1(\cF{,}\cF)
\,{\to}\,
\Ext^1(\pi^*\mbb L{,}\cF)
{\to}\,
\Ext^2(\cal Q_\cF{,}\cF)
\,{\to}\,0.
\end{eqnarray}
We deduce that $\Ext^2(\cal Q_\cF,\cF)\cong\Ext^2(\cal Q_\cF,\cal Q_\cF)$. 

\nit\textit{Claim}\quad $\Ext^2(\cal Q_\cF,\cal Q_\cF)=0$. \\ 
Indeed, the (reduced) support of $\cal Q_\cF$ is a finite, disjoint union of fibres of $\pi$, 
so we may assume that $\cal Q_\cF$ is supported on the thickening of a single fibre 
$\pi^{-1}(o)$, $o\in\mbb P^1$. 
In this case the annihilator ${\rm Ann}(\cal Q_\cF)\supset\langle x^m\rangle$, 
for some $m\ges 1$, that is $\cal Q_\cF$ is a $\mbb C[x]/\langle x^m\rangle$-module. 
For $m=1$, $\cal Q_\cF$ is a (torsion free) sheaf on $\pi^{-1}(o)$, 
so $\Ext^2(\cal Q_\cF,\cal S)=\Ext^2(\cal S,\cal Q_\cF)=0$ for any 
$\mbb C[x]/\langle x\rangle$-module $\cal S$. For the inductive step, 
notice that ${\rm Ann}(x\cal Q_\cF)\supset\langle x^{m-1}\rangle$, 
$\cal T{:=}\,\cal Q_\cF\bigl/x\cal Q_\cF\bigr.$ is a $\mbb C[x]/\langle x\rangle$-module, 
and we have the exact sequence $x\cal Q_\cF\subset\cal Q_\cF\surj\cal T$. 
We deduce the exact sequences 
\\[1ex] \centerline{\scalebox{.8}{$
\xymatrix@R=1.25em{
\Ext^2(\cal T,x\cal Q_\cF)\ar[d]&&\Ext^2(\cal T,\cal T)\ar[d]
\\ 
\Ext^2(\cal Q_\cF,x\cal Q_\cF)\ar[r]\ar[d]
&\Ext^2(\cal Q_\cF,\cal Q_\cF)\ar[r]
&\Ext^2(\cal Q_\cF,\cal T)\ar[d]
\\ 
\Ext^2(x\cal Q_\cF,x\cal Q_\cF)&&\Ext^2(x\cal Q_\cF,\cal T),
}
$}}\\[1ex] 
and apply the induction hypotheses. 
Similarly, $\Ext^2(\cal Q_{\cF},\cal S)=\Ext^2(\cal S,\cal Q_{\cF})=0$ holds, 
for any $\mbb C[x]/\langle x\rangle$-module $\cal S$. 
\end{proof}

The previous lemmas show that any $L_c$-semi-stable $\cF$ fits into an exact sequence 
\begin{equation}\label{eq:QL}
0\to\pi^*\mbb L\srel{f}{\to}\cF\srel{q}{\to}\cal Q\to0,
\end{equation}
with  $\mbb L\,{:=}\ouset{j=1}{p}{\mbox{$\bigoplus$}}\cO(-a_j)^{\oplus r_j}, 
a_j{\ges}\,1$, as in \eqref{eq:a}, and $Q$ as in \ref{lm:QQ}. 
The homomorphism $q$ is the (canonically defined) quotient map, while 
$f\,{\in}\,\Hom(\pi_*\mbb L,\cF)\,{=}\,\End(\mbb L)$ is defined modulo $\Aut(\mbb L)$. 
The equivalence classes of sequences \eqref{eq:QL} are parameterized by 
$\Ext^1(\cal Q,\pi^*\mbb L)$, where $(f,q)$ is equivalent to 
$(\vphi f,q\vphi^{-1})$ for all $\vphi\in\Aut(\cF)$. So there is an 
$\Aut(\mbb L)/\Aut(\cF)$-ambiguity in defining $f$ in \eqref{eq:QL}. 

\begin{remark}\label{rk:LF}
$\dim\End(\mbb L)\ges r^2$ and $\ext^1(\pi^*\mbb L,\cF)\ges r(n-n_\cF)$. 
Equality holds (in both places) if and only if 
\begin{equation}\label{eq:generic}
\mbb L=\mbb L_{n_\cF,r}{:=}\cO_{\mbb P^1}(-a_1)^{\oplus r_1}{\oplus}
\cO_{\mbb P^1}(-a_2)^{\oplus r_2} 
\;\text{with}\;
\left\{\kern-1ex
\begin{array}{ll}
a_1=\lfloor\frac{n_\cF}{r}\rfloor, 
&
r_1=r+r\cdot\lfloor\frac{n_\cF}{r}\rfloor-n_\cF;
\\[1.5ex] 
a_2=\lfloor\frac{n_\cF}{r}\rfloor+1, 
&
r_2=n-r\cdot\lfloor\frac{n_\cF}{r}\rfloor.
\end{array}
\right.
\end{equation}
(Here $\lfloor\,\cdot\,\rfloor$ stands for the integral part. 
For $n_\cF$ divisible by $r$, the $a_2$-term is missing.) Indeed, 
$\ext^0(\pi^*\mbb L,\cF)-\ext^1(\pi^*\mbb L,\cF)=r^2-r(n-n_\cF)$, and we notice that 
$$
\begin{array}{l}
\dim\End\biggl(
\ouset{j=1}{p}{\mbox{$\bigoplus$}}\cO(-a_j)^{\oplus r_j}
\kern-.7ex
\biggr)
\kern-.5ex
=
\kern-.5ex
\ouset{i'\ges i}{}{\mbox{$\sum$}}\,r_{i'}r_i(a_{i'}-a_i+1)
\ges
\ouset{i}{}{\mbox{$\sum$}}\;r_i^2
+2\ouset{i'> i}{}{\mbox{$\sum$}}\;r_{i'}r_i
\!
=
{\bigl(\ouset{i}{}{\mbox{$\sum$}}\;r_i\bigr)}^2
\kern-1ex
=r^2.
\end{array}
$$
Equality occurs precisely when the sequence $a_1\,{<}a_2\,{<}...$ contains at most 
two integers such that $a_2-a_1=1$. 
In this case $a_1,a_2,r_1,r_2$ are as in \eqref{eq:generic}. 
\end{remark}

\subsection{Basic properties of the moduli space of semi-stable sheaves on $Y_\ell$}

\begin{lemma}\label{lm:LF}
{\rm(i)} The dimension of the locally closed subset 
\begin{equation}\label{eq:LF}
M_{\mbb L}:=
\{ \cF\mid \cF\text{ fits into } \eqref{eq:QL} \} = \{ \cF \mid \pi_*\cF\cong\mbb L \}
\subset\bar M_{Y_\ell}^{L_c}(r;0,n),
\end{equation}
is at most $\ext^1(\cF,\cF)-r(n-n_\cF)$.

\nit{\rm(ii)} 
$\{ \cF\in M_{\mbb L}\mid \det(f)^\vee\text{ has simple zeros}\}
\subset M_{\mbb L}$ is dense. 
Thus the irreducible components of $M_{\mbb L}$ are 
$M_{\mbb L,\bsymb{b}}:=
\overline{\{
\cF\in M_{\mbb L}\mid \cF/\pi^*\mbb L\cong\cal Q
\text{ of the form \ref{lm:QQ}(iv)}\}}^{M_{\mbb L}}\!$, 
with  $\bsymb{b}=(b_1,\ldots,b_{n_\cF})$ as in \ref{lm:QQ}(iv).

\nit{\rm(iii)} If $r\ges 2$, then the generic $\cF\in M_{\mbb L,\bsymb{b}}$ 
is locally free, for all $(\mbb L,\bsymb{b})$ as above. (This statement is false for 
$r=1$, unless $n=0$.)
\end{lemma}

\begin{proof} 
(i) The infinitesimal deformations of $\cF$ induced by deformations of 
$\pi^*\mbb L\srel{f}{\to}\cF$, where \textit{both} $\cF$ and $f$ vary, are given by 
$\Img\big(\!\Ext^1(\cal Q,\cF){\to}\Ext^1(\cF,\cF)\big)
\!{\srel{\eqref{eq:s2}}{\cong}}\!
\Ext^1(\cal Q,\cF)\bigl/\big(\!\End(\mbb L)/\End(\cF)\big)\bigr.$,\footnote{
Maybe is enlightening to outline the analytic proof of this statement 
for $\cF$ locally free. Let $F$ be the subjacent $\cal C^\infty$ vector bundle. 
(This is the same for all holomorphic deformations of $\cF$.) 
A holomorphic structure $\cF_o$ in $F$ is determined by  
$\delta_o:\cal C^\infty(F)\to\Omega^{0,1}(F)$  satisfying the Leibniz rule, 
and $\delta_o^2=0$. A deformation $\cF_\veps=(F,\delta_\veps)$ of $\cF_o$ is given 
by $\alpha\in\Omega^{0,1}(End_{\cal C^\infty}(F))$ such that $\delta_o\alpha=0$, 
that is $\alpha\in H^1(End(\cF_o))$. 
A deformation $\pi^*\mbb L\srel{f_\veps}{\to}\cF_\veps=(F,\delta_\veps)$ is 
given by $(\alpha,\phi)\in\Omega^{0,1}(End(F))\oplus\cal C^\infty(\pi^*\mbb L,F)$, 
such that $\delta_o\alpha=0$, $\delta_o\phi=\alpha f_o$. This is equivalent to saying that 
$$
\alpha\in\Ker\big(\;H^1(End(\cF_o))\srel{f_o}{\to}H^1(Hom(\pi^*\mbb L,\cF_o))\;\big)
=\Img\big(\;\Ext^1(\cal Q_o,\cF_o)\to \Ext^1(\cF_o,\cF_o)\;\big).
$$
In this case, $\phi$ is determined up to some $\psi\in\End(\mbb L)$. 
However, these choices induce trivial deformations of $\cF_o$.} 
and 
\\ \centerline{$
\xymatrix@R=-1ex{
&\Ext^1(\cF,\cF)\\ 
\Ext^1(\cal Q,\cF)\ar[ru]^-{\rd_\cF}\ar[rd]_-{\rd_{\cal Q}}&\\ 
&\Ext^1(\cal Q,\cal Q)}
$}
are the differentials of $[\pi^*\mbb L{\srel{f}{\to}}\cF]\,{\mt}\cF$ and 
$[\pi^*\mbb L{\srel{f}{\to}}\cF]\,{\mt}\cal Q\,{=}\,\Coker(f)$ respectively. 
(Notice that ${\rd_{\cal Q}}$ passes to the quotient, since  
$\End(\mbb L)\subset\Ker(\rd_{\cal Q})$.) The sequence \eqref{eq:s2} implies: 
\\ \centerline{
$\ext^1(\cal Q,\cF)-\dim\big(\End(\mbb L)/\End(\cF)\big)\les\ext^1(\cF,\cF)-r(n-n_\cF).$
}
\nit(ii) \textit{Step 1}\quad Apply $\Hom(\cal Q,\cdot)$ to 
$\cF(-\Lambda)\subset\cF\surj\cF_\Lambda$, and deduce 
$\Ext^1(\cal Q,\cF)\to\Ext^1(\cal Q_{\Lambda},\cF_\Lambda)$ is surjective. 
Thus the infinitesimal deformations of $\mbb L\to\cF_\Lambda$ on $\Lambda$ 
come from infinitesimal deformations of $\pi^*\mbb L\to\cF$ on $Y_\ell$. 
(Notice that both deformations are unobstructed, as the corresponding $\Ext^2$ 
groups vanish.)

\nit\textit{Step 2}\quad 
Apply $\Hom(\cal Q_{\Lambda},\cdot)$ to 
$\mbb L\subset\cF_\Lambda\surj\cal Q_{\Lambda}$ 
($\L$ is $\cal Q_\cF$-regular, by \ref{lm:QQ}(i)), and deduce that 
$\Ext^1(\cal Q_{\Lambda},\Lambda)\to\Ext^1(\cal Q_{\Lambda},\cal Q_{\Lambda})$ 
is surjective. Thus, the deformations of $\cal Q_{\Lambda}$ on $\Lambda$ 
(they are unobstructed) can be lifted to deformations of $\mbb L\to\cF_\Lambda$. 

\nit\textit{Step 3}\quad $\cal Q_{\Lambda}$ is a sheaf of length $n_\cF$ on $\mbb P^1$, 
and its generic deformation is the structure sheaf of $n_\cF$ distinct points on $\mbb P^1$. 
By our previous discussion, this deformation is induced by a deformation 
$\pi^*\mbb L\srel{f'}{\to}\cF'$ of $\pi^*\mbb L\srel{f}{\to}\cF$ on $Y_\ell$, 
and the quotient $\cal Q'$ is of the form \ref{lm:QQ}(iv). 

\nit{\rm(iii)} By dualizing \eqref{eq:QL}, the statement is equivalent to the surjectivity 
of the generic homomorphism 
$\pi^*\mbb L^\vee{\srel{e}{\to}}\,
\ouset{i=1}{n_\cF}{\bigoplus}\cO_{\pi^{-1}(x_i)}(b_i)$. 
This is true, since $e$ factorizes 
\\ \centerline{
$\pi^*\mbb L^\vee{\to}\,
\ouset{i=1}{n_\cF}{\bigoplus}\pi^*\mbb L^\vee_{\pi^{-1}(x_i)}
\cong
\ouset{i=1}{n_\cF}{\bigoplus}\cO_{\pi^{-1}(x_i)}^{\oplus r}{\to}\,
\ouset{i=1}{n_\cF}{\bigoplus}\cO_{\pi^{-1}(x_i)}(b_i)$,
} 
and each 
$\cO_{\pi^{-1}(x_i)}(b_i)$ can be generated by two sections ($r\ges2$). 
\end{proof}

\begin{theorem}\label{thm:hirz}
Let $\bar M^{L_c}_{Y_\ell}(r;0,n)$ be the moduli space of (equivalence classes of) 
$L_c$-semi-stable torsion free sheaves on $Y_\ell$, with $c>r(r-1)n$. 
We denote by $M^{L_c}_{Y_\ell}(r;0,n)^\vb$ the open subset corresponding 
to $L_c$-stable vector bundles. Then the following statements hold:
\begin{enumerate}
\item 
$\bar M^{L_c}_{Y_\ell}(r;0,n)$ is the  union of the irreducible, locally closed strata 
$M_{\mbb L,\bsymb{b}}$ \ref{lm:LF}(ii). For any $(\mbb L,\bsymb{b})$, 
the generic $\cF\in M_{\mbb L,\bsymb{b}}$ satisfies: 
\begin{enumerate}
\item $\cF_\Lambda\cong\cO_{\Lambda}^{\oplus r}$, 
$\cF_\l\cong\cO_\l^{\oplus r}$ for generic $\l\in|\cO_\pi(1)|$;
\item is locally free, for $r\ges 2$. 
\end{enumerate}
Conversely, any vector bundle $\cF$ with this property is $L_c$-semi-stable. 
\item 
$M_{\mbb L_{n{,}r}}\!$ is the \emph{unique} top dimensional, open stratum, 
which corresponds to sheaves $\cF$ with $\pi_*\cF\cong\mbb L_{n,r}$. 
\item 
The generic point $\cF\in M_{\mbb L_{n,r}}$ satisfies 
$\cal Q_\cF\,{=}\ouset{i=1}{n}{\mbox{$\bigoplus$}}\cO_{\pi^{-1}(x_i)}(-1)$ with 
$\{x_1,\ldots,x_n\}\subset\mbb P^1$ pairwise distinct, 
$\cF_\Lambda\cong\cO_{\Lambda}^{\oplus r}$, and is $L_c$-stable. 
Hence $M^{L_c}_{Y_\ell}(r;0,n)^\vb$ is non-empty. 
\item 
For $r\ges 2$, $M^{L_c}_{Y_\ell}(r;0,n)^\vb$ is smooth, of dimension 
$2rn-r^2+1$, and dense in $\bar M^{L_c}_{Y_\ell}(r;0,n)$.
\end{enumerate}
\end{theorem}

\begin{proof}
\nit(i)+(ii) Everything is proved in lemma \ref{lm:LF}, except that $\cF_\Lambda$ is 
trivializable. We know that the points $\cF\in M_{\mbb L,\bsymb{b}}$ with 
$\cal Q_\cF$ of the form \ref{lm:QQ}(iv) are dense. 
For $\cal Q$ of this form, the generic $\cF_\Lambda\in\Ext^1(\cal Q_\Lambda,\mbb L)$ 
is the trivial vector bundle on $\Lambda$, 
and $\Ext^1(\cal Q,\pi^*\mbb L)\to\Ext^1(\cal Q_\L,\mbb L)$ is surjective. 
For proving that the generic $\cF\in M_{\mbb L,\bsymb{b}}$ is trivializable along 
the generic $\l$, we notice that $\l$ is a flat deformation of $(\L+\ell$ fibres of $\pi)$. 
But $\cF_\L$ and $\cF_\fbr$ are both trivializable, and the claim follows.

Conversely, if $\cF_\L$ and $\cF_\fbr$ are both trivial, then 
$\deg_\L\cG\les 0,\deg_{\fbr}\cG\les 0$, for any saturated subsheaf $\cG\subset\cF$, 
so $\cF$ is $L_c$-semi-stable, indeed. 
\smallskip 

\nit(iii) 
We should prove that the generic $\cF$ is $L_c$-stable. Otherwise it admits a proper, 
stable subsheaf $\cG\subset\cF$ such that $\deg(\cG)=\deg(\cF/\cG)=0$, both 
$\cG,\cF/\cG$ are torsion free, and $\cF/\cG$ is semi-stable. ($\cG$ is the first term of 
the Jordan-H\"older filtration of $\cF$.) For shorthand, we denote 
$\rk(\cG)=r'$ and $c_2(\cG)=n'$. As before, $n'\ges r'$ and also: \\ 
\centerline{$\chi(\cF)=\chi(\cG)+\chi(\cF/\cG)
\;\Rightarrow\;
n=c_2(\cF)=c_2(\cG)+c_2(\cF/\cG)=n'+c_2(\cF/\cG).$ }
Since $\cG,\cF/\cG$ are semi-stable, the Bogomolov inequality \cite[Theorem 3.4.1]{hl} 
implies $0\les n'\les n$. 

\nit\textit{{Claim}} The dimension of the infinitesimal deformations of $\cF$ 
is strictly larger than that of $\cG$: 
\begin{equation}\label{eq:ineqFG}
2rn-r^2+1>2r'n'-(r')^2+1\;\Leftrightarrow\;
(n-n')(n+n')>\bigl(n-n'-(r-r')\bigr)\bigl(n+n'-(r+r')\bigr).
\end{equation}
(For the left hand side we used $\ext^0(\cG,\cG)=1$ and $\ext^2(\cG,\cG)=0$, 
as $\cG$ is stable.) The latter inequality is indeed satisfied: \\ 
\centerline{$
(n-n')(n+n')\srel{r-r'\ges 1}{>}\bigl(n-n'-(r-r')\bigr)(n+n')
\srel{(*)}{\ges}
\bigl(n-n'-(r-r')\bigr)\bigl(n+n'-(r+r')\bigr).
$} 
Concerning $(*)$: if $n-n'-(r-r')\les 0$, then everything is fine, since in 
\eqref{eq:ineqFG} the right hand side is negative. This proves the claim. 

We obtained a contradiction: on one hand, the generic $\cF$ is properly semi-stable, 
while, on the other hand, the possible de-semi-stabilizing subsheaves have strictly lower 
deformation space. This proves that the generic $\cF$ is indeed $L_c$-stable. 
\smallskip

\nit(iv) For any stable $\cF\in\bar M^{L_c}_{Y_\ell}(r;0,n)$, the Riemann-Roch 
formula yields  $\chi(\End(\cF))=-2rn+r^2$, while the stability of $\cF$ implies 
$h^0(\End(\cF))=1, h^2(\End(\cF))=0$. 
Thus $M^{L_c}_{Y_\ell}(r;0,n)^\vb$ is smooth, of dimension $2rn-r^2+1$. 
On the other hand, all the strata $M_{\mbb L}$, with $\mbb L\neq\mbb L_{n,r}$, 
are strictly lower dimensional. 
\end{proof}

\begin{theorem}\label{thm:hilb}
There is a well-defined surjective morphism 
\begin{equation}\label{eq:mor}
h:\bar M^{L_c}_{Y_\ell}(r;0,n)\to{\rm Hilb}^n_{\mbb P^1}\cong\mbb P^n, 
\quad \cF\mt\Supp R^1\pi_*\cF(-\Lambda).
\end{equation} 
Its generic fibre is $(2rn-r^2-n+1)$-dimensional, the quotient of an open set in 
$\mbb A_{\mbb C}^{2nr}$ by the linear action of a $(r^2+n-1)$-dimensional group. 
\end{theorem}

\begin{proof}
We saw that $\deg_{\mbb P^1}R^1\pi_*\cF(-\Lambda)=n$, for all $\cF$, so $h$ is 
well-defined set theoretically. Actually, there is a technical detail: 
$\bar M^{L_c}_{Y_\ell}(r;0,n)$ parameterizes equivalence classes of $L_c$-semi-stable 
sheaves where $\cF$ and $\cF'$ are equivalent if their Jordan-H\"older factors are 
isomorphic. Thus we must prove that if 
${\rm grad}^{\rm JH}(\cF)\cong{\rm grad}^{\rm JH}(\cF')$, 
then $R^1\pi_*\cF(-\L)\cong R^1\pi_*\cF'(-\L)$. For this, observe that if $\cF$ fits 
into $0\to\cF_1\to\cF\to\cF_2\to 0$, with $\cF_1,\cF_2$ semi-stable (of degree $0$), 
then $0\to R^1\pi_*\cF_1(-\L)\to R^1\pi_*\cF(-\L)\to R^1\pi_*\cF_2(-\L)\to0$. 
The conclusion follows now by induction on the length of the Jordan-H\"older filtration.

Returning to the theorem, we prove that $h$ satisfies the functorial property of 
${\rm Hilb}^n_{\mbb P^1}$. For a scheme $S$, we denote $Y_{S}:=S\times Y_\ell$, 
$\mbb P^1_S:=S\times\mbb P^1$, $\pi_S:=({\rm id}_S,\pi)$, \textit{etc}, 
and $\cO(1):=\cO_\pi(1)+\cal O_{\mbb P^1}(1)$.\smallskip 

\nit\textit{{Claim}}\quad Let $\cG$ be a torsion free sheaf on $Y_S$, such that 
$\pi_{*}\cG_s=0$, for all $s\in S$ (thus $\pi_{S*}\cG=0$, too). Then the natural 
homomorphism $\gamma_s:R^1\pi_{S*}\cG\otimes\cO_{S,s}\to R^1\pi_{s*}\cG_s$ 
is an isomorphism, for all $s\in S$, and the sheaf $R^1\pi_{S*}\cG$ on $\mbb P^1_S$ 
is $S$-flat. 

Since flatness is a local property, is enough to prove the statement for $S=\Spec(A)$, 
where $(A,\mfrak m)$ is a local ring, and $s=\Spec(A/\mfrak m)\in S$. 
\smallskip

Any $\cG$ admits a finite resolution 
$\ldots{\to}\,\mbb L_2{\to}\,\mbb L_1{\to}\,\mbb L_0{\to}\,\cG{\to}\,0,$ 
with $\mbb L_j{=}\,\cO(-c_j)^{\oplus m_j}$. 
We prove the claim by induction on the length of the resolution. 
If the length is zero,  that is $\cG{=}\,\cO(-c)^{\oplus m}$ 
(notice that $\pi_*\cG_s=0$ implies $c\,{>}\,0$), then $R^1\pi_{S*}\cG$ 
is locally free on $\mbb P^1_S$, and $\gamma_s$ is an isomorphism. 

For the inductive step, we fit $\cG$ into $0\to\cG'\to\mbb L\to\cG\to0,$ where 
$\mbb L{=}\,\cO(-c)^{\oplus m},\, c>0$, $\cG'$ admits a resolution of 
length one less, and also $\pi_*\cG'_s=0$. 
By the hypothesis, $\gamma'_s$ (for $\cG'$) is an isomorphism, which immediately 
yields that $\gamma_s$ (for $\cG$) is an isomorphism too. Also, from the exact sequence 
$0{\to}\,R^1\pi_{S*}\cG'{\to}\,R^1\pi_{S*}\mbb L{\to}\,R^1\pi_{S*}\cG{\to}\,0$ 
we deduce the equivalences: \\ 
\centerline{$
R^1\pi_{S*}\cG\text{ is $A$-flat}
\Leftrightarrow
\Tor^A_1\Bigl(R^1\pi_{S*}\cG,\frac{A}{\mfrak m}\Bigr)=0
\ouset{\lower3pt\hbox{\scriptsize is $A$-flat}}%
{\raise3pt\hbox{\scriptsize$R^1\pi_{S*}\mbb L$}}%
{\Longleftrightarrow}
\underbrace{R^1\pi_{S*}\cG'{\otimes_A}\frac{A}{\mfrak m}}_%
{\underset{\gamma'_s}{\cong}\; R^1\pi_{*}\cG'_s}
\kern3pt
\ouset{\lower1pt\hbox{\scriptsize injective}}{}
{-\kern-7pt-\kern-7pt-\kern-7pt\longrightarrow}
\kern3pt
\underbrace{R^1\pi_{S*}\mbb L{\otimes_A}\frac{A}{\mfrak m}}_%
{\cong\; R^1\pi_{*}\mbb L_s}.
$}\\[.5ex] 
The homomorphism on the right hand side is indeed injective, because $\pi_*\cG_s=0$. 

Now we apply the claim to our setting. For a torsion free sheaf $\cF$ on $Y_S$ which 
is $S$-fibrewise $L_c$-semi-stable, we have $\pi_{*}\cF(-\Lambda_S)_s=0$, 
so $R^1\pi_{S*}\cF(-\Lambda_S)$ is $S$-flat. Hence $h$ is a morphism, as desired. 
Its generic fibre is the quotient of an open subset of 
$\Ext^1\Big(\ouset{i=1}{n}{\bigoplus}
\cO_{\pi^{-1}(x_i)}(-1),\pi^*\mbb L_{n,r}\Big)$, 
which is $2rn$-dimensional, by the action of 
$\bigl(\Aut(\mbb L_{n,r})\times(\mbb C^*)^n\bigr)\bigl/\mbb C^*\bigr.$. 
Since $h$ is dominant and  $\bar M_{Y_\ell}^{L_c}(r;0,n)$ is projective, 
we deduce that $h$ is surjective. 
\end{proof}


\subsection{Rationality issues}\label{ssct:rtl}

We conclude this section with a \emph{self-contained} proof of the rationality of 
$\bar M^{L_c}_{Y_\ell}(r;0,n)$, and some applications. 
This result is proved in \cite{c-mr} for arbitrary $c_1$, under the assumption that 
the discriminant $\Delta(\cF)$ is very large. However, no explicit bounds are given.

\begin{theorem}\label{thm:rtl}
$\bar M^{L_c}_{Y_\ell}(r;0,n)$ is a rational variety, for all $n\ges r\ges2$. 
\end{theorem}

\begin{proof} 
Is enough to prove that 
$M^o\,{:=}\,\{ \cF\in M_{\mbb L_{n,r}}\,{\mid}\,\det(f)^\vee\text{ has simple zeros}\}$ 
is rational. Any $\cF\in M^o$ fits into 
\\ \centerline{
$0\to
\pi^*\mbb L=\pi^*\Big(
\cO_{\mbb P^1}(-a_1)^{\oplus r_1}\oplus\cO_{\mbb P^2}(-a_1-1)^{\oplus r_2}
\Big)
\to
\cF
\to
\ouset{i=1}{n}{\bigoplus}\cO_{\pi^{-1}(x_i)}(-1)
\to0,
$ 
} 
with $\bsymb{x}:=\{x_1,\ldots,x_n\}\subset\mbb P^1$ pairwise distinct, and 
$a_1,r_1,r_2$ given by \eqref{eq:generic}. For given $\bsymb{x}$ these extensions 
are parameterized by \vspace{-1ex}
\begin{equation}\label{eq:Ex}
\cal E_{\bsymb{x}}:=\ouset{i=1}{n}{\bigoplus}\;
\mbb L_{x_i}
{\otimes}\,
\overbrace{\Gamma(\cO_{\pi^{-1}(x_i)}(1))}^{\cong\;\mbb C^2}
=\Big(\mbb L\otimes\pi_*\cO_\pi(1)\Big)\otimes\cO_{\bsymb{x}}, 
\quad\dim\cal E_{\bsymb{x}}=2nr.
\end{equation}
This space is acted on by 
$$
G_{\bsymb{x}}:=
\bigl(\Aut(\mbb L)\times(\mbb C^*)^n\bigr)\bigl/\mbb C^*\bigr.
{=}\,
\bigl(\Aut(\mbb L)\times\cO_{\bsymb{x}}^\times\bigr)\bigl/\mbb C^*\bigr.
\quad
(\cO_{\bsymb{x}}^\times\subset\cO_{\bsymb{x}}
\text{ stands for the invertible elements}), 
$$ 
so the fibre 
$M^o_{\bsymb{x}}=h^{-1}(\bsymb{x})\cap M^o$ is the quotient by $G_{\bsymb{x}}$ 
of an open subset of the affine space underlying $\cal E_{\bsymb{x}}$. 
The symbol `$\invq\,$' will always stand for the quotient of some open subset. 

In order to globalize this construction as $\bsymb{x}\in\Hilb^n_{\mbb P^1}$ varies, 
we consider the diagram: 
\begin{equation}\label{eq:hilb}
\xymatrix@C=3em@R=1.5em{
\cal X\ar[r]^-{q'}\ar[d]_-\xi
&
\cal Z\;\ar[d]_-{\zeta}\ar@{^(->}[r]\ar[rd]^(.7){\pr_{\mbb P^1}}
&
\Hilb^n_{\mbb P^1}\times\mbb P^1\ar[d]
&
Y_\ell\ar[ld]^-{\pi}
\\ 
{(\mbb P^1)}^n\ar[r]^-q
&
{(\mbb P^1)}^n/\mfrak S_n=\mbb P^n\cong\Hilb^n_{\mbb P^1}&\mbb P^1&
}
\end{equation}
Here $\mfrak S_n$ stands for the group of permutations of $n$ elements. 
Since we are interested in birational properties, we will repeatedly restrict ourselves 
to appropriate open subsets; they will be denoted by $\cal U\subset (\mbb P^1)^n$ 
and $\cal H:=\cal U/\mfrak S_n\subset \Hilb^n_{\mbb P^1}$. 
We start by restricting ourselves to the complement $\cal U$ of the diagonals. 
(Thus $\mfrak S_n$ acts freely, and $q$ is flat.) 
In algebraic terms, \eqref{eq:hilb} reduces to 
\begin{equation}\label{eq:alg}
\xymatrix@C=3em@R=1.5em{
\disp\frac{\mbb C[x_1,\ldots,x_n][z]}{\langle (z-x_1)\cdot\ldots\cdot(z-x_n)\rangle}
&\disp
\frac{\mbb C[s_1,\ldots,s_n][z]}{\langle z^n-s_1z^{n-1}+s_2z^{n-2}-\ldots\rangle}
\ar[l]
&
\mbb C[s_1,\ldots,s_n][z]\ar[l]
\\ 
\mbb C[x_1,\ldots,x_n]\ar[u]
&\mbb C[s_1,\ldots,s_n]\ar[u]\ar[l]_-{\text{inclusion}}
&\mbb C[z],\ar[ul]\ar[u]
}
\end{equation}
where $s_1=x_1+\ldots+x_n,\ldots,s_n=x_1\cdot\ldots\cdot x_n$ 
are the symmetric polynomials. 

In this setting,  \eqref{eq:Ex} is the stalk at $\bsymb{x}$ of  
\begin{equation}\label{eq:E}
\cal E:=
\zeta_*\big(\pr_{\mbb P^1}^*(\mbb L\otimes\pi_*\cO_{\pi}(1))\big)
=
\zeta_*\Big(\pr_{\mbb P^1}^*\big(\,
\mbb L\otimes(\cO_{\mbb P^1}\oplus\cO_{\mbb P^1}(\ell))
\,\big)\Big),
\end{equation} 
and the symmetry group acting on $\cal E$ is 
\begin{equation}\label{eq:GG}
G:=\big(\Aut(\mbb L)\times(\zeta_*\cO_{\cal Z})^\times\big)/\mbb C^*. 
\end{equation}
Notice that $\zeta_*\cO_{\cal Z}$ is a sheaf of algebras on $\cal H$, so it 
makes sense to consider the (multiplicative) subgroup of invertible elements 
$(\zeta_*\cO_{\cal Z})^\times$. 

We simplify $\cal E$ by shrinking $\cal U$ and $\cal H=\cal U/\mfrak S_n$ further. 
Indeed, fix $\infty=\langle0,1\rangle\in\mbb P^1$ and trivialize the various 
$\cO_{\mbb P^1}(a)$, $a\in\mbb Z$, appearing in \eqref{eq:E} on the 
complement $\mbb A^1=\mbb P^1\sm\{\infty\}$. Moreover, we fix two general sections 
$\si_0,\si_1$ in $\cO_\pi(1)$. Their zero loci intersect at $\ell$ points in $Y_\ell$, 
lying above $\{u_1,\ldots,u_\ell\}\subset\mbb A^1$. For $x\neq u_j$, the restrictions 
$\si_{0,x},\si_{1,x}\in\Gamma(\pi^{-1}(x),\cO_\pi(1))$ yield a basis. 
Then take $\cal U$ to consist of $(x_1,\ldots,x_n)\in\mbb A^n\subset(\mbb P^1)^n$ 
pairwise distinct, such that $x_i\neq u_j$ for $i=1,\ldots,n$, $j=1,\ldots,\ell$. 
One can trivialize $\cal E$ over $\cal H=\cal U/\mfrak S_n$ as follows: 
\begin{equation}\label{eq:E-triv}
\begin{array}{l}
\cal E=\zeta_*\big(
\pr_{\mbb P^1}^*(\cO^r\otimes\mbb C^2)
\big)
=
(\zeta_*\cO_{\cal Z}\otimes\mbb C^{r})_\lft
\oplus
(\zeta_*\cO_{\cal Z}\otimes\mbb C^{r})_\rgt,
\\ 
\zeta_*\cO_{\cal Z}\srel{\eqref{eq:alg}}{=}
\mbb C[s_1,\ldots,s_n]\oplus\ldots\oplus\hat z^{n-1}\cdot\mbb C[s_1,\ldots,s_n]
\cong\cO_{\cal H}^n.
\end{array}
\end{equation}
The subscripts refer to the factors $\Aut(\mbb L)$,  $(\zeta_*\cO_{\cal Z})^\times$ 
of $G$, respectively. Although they act simultaneously, 
$(\cO_{\cal H}^n\otimes\mbb C^r)_\lft$ will be viewed (mainly) as an 
$\Aut(\mbb L)$-module, while $(\cO_{\cal H}^n\otimes\mbb C^r)_\rgt$ 
will be (mainly) a $(\zeta_*\cO_{\cal Z})^\times$-module. Let 
\\ \centerline{
$E:=\Spec_{\Hilb^n_{\mbb P^1}}(\Sym^\bullet\cal E^\vee)
\cong\cal H\times\mbb A^{rn}_\lft\times\mbb A^{rn}_\rgt 
$ 
}
be the linear fibre space (quasi-projective variety) determined by $\cal E$. We write 
$\mbb L=\bigl(\,\mbb C^{r_1}\otimes\cO_{\mbb P^1}\oplus
\mbb C^{r_2}\otimes\cO_{\mbb P^1}(-1)\,\bigr)
\otimes\cO_{\mbb P^1}(-a_1)$, and think off the elements 
of $\mbb L_x$, for $x\in\mbb P^1$, as column vectors with $r=r_1+r_2$ entries. 
Then the elements of $E_{\bsymb{x}}$, $\bsymb{x}\in\cal H$, can be represented 
as pairs of $r\times n$-matrices in the block form
\begin{equation}\label{eq:xi}
\scalebox{.8}{$
\bsymb{e}=
\left[
\begin{array}{ccc|ccc|ccc|c}
*&\ldots&*&*&\ldots&*&*&\ldots&*&\ldots
\\ 
&[\text{I}]_{r_1\times r_1}&&
&[\text{III}]_{r_1\times r_2}&&&[\text{V}]_{r_1\times r_2}
&&
\\ 
*&\ldots&*&*&\ldots&*&*&\ldots&*&\ldots 
\\ \hline 
*&\ldots&*&*&\ldots&*&*&\ldots&*&\ldots
\\ 
&[\text{II}]_{r_2\times r_1}&&
&[\text{IV}]_{r_2\times r_2}&&&[\text{VI}]_{r_2\times r_2}
&&
\\ 
*&\ldots&*&*&\ldots&*&*&\ldots&*&\ldots 
\end{array}
\right].
$}
\end{equation}

The strategy for proving the rationality of $M^o$ is to exhibit a subvariety (Luna slice) 
$S\subset E$ which is a locally trivial, linear fibre bundle over 
$\cal H\subset\Hilb^n_{\mbb P^1}$, and the restriction of $E\dashto E\invq G$ to $S$ 
is birational. (In the terminology of \cite[Definition 2.9]{re}, 
$S$ will be a $(G,\bone)$-section of $E\to\cal H$.) 
The slice will be constructed by proving that the \emph{generic} pair 
of matrices \eqref{eq:xi} admits a unique (suitable) canonical form. 


\subsubsection{The $\Aut(\mbb L)$-action}\label{sssct:L}  
Now we turn our attention to the $\Aut(\mbb L)$-action on $E_{\bsymb{x}}$. 
First, we observe that the elements of $\Aut(\mbb L)$ can be represented schematically 
as follows: 
\begin{equation}\label{eq:group}
\begin{array}{rl}
\Aut(\mbb L)=& 
\left\{\;\left[
\begin{array}{c|c}
A\in\Gl(r_1;\mbb C)&
H(z)\in\Hom(\mbb C^{r_2},\mbb C^{r_1})\otimes\Gamma(\cO_{\mbb P^1}(1))
\\
\hline 
0&B\in\Gl(r_2;\mbb C)
\end{array}
\right]\;\right\}, 
\\[2.5ex] 
\text{with}\;& 
\Hom(\mbb C^{r_2},\mbb C^{r_1})\otimes\Gamma(\cO_{\mbb P^1}(1))=
\{H(z)=z^{(0)}H_0+z^{(1)}H_1\mid H_0,H_1\in\Hom(\mbb C^{r_2},\mbb C^{r_1})\}.
\end{array}
\end{equation}
Since we restricted ourselves to a subset of $\mbb A^1\subset\mbb P^1$, 
we write $H(z)=H_0+zH_1$. Let us consider 
\\ \centerline{ 
$\bsymb{e}=\left[\begin{array}{c}u_0\\ v_0\end{array}\right]
+\hat z\left[\begin{array}{c}u_1\\ v_1\end{array}\right]
+\ldots+
\hat z^{n-1}\left[\begin{array}{c}u_{n-1}\\ v_{n-1}\end{array}\right]
\in\mbb A^{rn}_\lft,$
}
where the columns refer to the splitting $r=r_1+r_2$ in $\mbb L$. 
Some calculations show that 
\\ \centerline{
$g=\left[\begin{array}{c|c}
A&H_0+zH_1 \\ \hline 0&B
\end{array}\right]\in\Aut(\mbb L)$
}
acts on $\bsymb{e}$ as follows ($v_{-1}:=0$):
\begin{equation}\label{eq:ge}
g\times\bsymb{e}=
\sum_{j=0}^{n-1}\;\hat z^j\cdot\left[
\begin{array}{c}
Au_j+H_0v_j+H_1\big(v_{j-1}+(-1)^{n-j-1}s_{n-j}v_{n-1}\big)\\[1ex] Bv_j\kern5ex
\end{array}
\right].
\end{equation}
In the block form \eqref{eq:xi}, it reads:
$$
\begin{array}{rcl}
g\times\text{[I]}&=&
A\big[u_0,\ldots,u_{r_1-1}\big]
+H_0\big[v_0,\ldots,v_{r_1-1}\big]
\\[.5ex] && 
+H_1\big[(-1)^{n-1}s_n\cdot v_{n-1},\ldots,
v_{r_1-2}+(-1)^{n-r_1}s_{n-r_1+1}\cdot v_{n-1}\big]
\\[.5ex]&=&
A\text{[I]}+H_0\text{[II]}+H_1\text{[II']},
\\[1.5ex]
g\times\text{[III]}&=&
A\big[u_{r_1},\ldots,u_{r-1}\big]
+H_0\big[v_{r_1},\ldots,v_{r-1}\big]
\\[.5ex] &&
+H_1\big[v_{r_1-1}+(-1)^{n-r_1-1}s_{n-r_1}\cdot v_{n-1},\ldots,
v_{r-2}+(-1)^{n-r}s_{n-r+1}\cdot v_{n-1}\big]
\\[.5ex]&=&
A\text{[III]}+H_0\text{[IV]}+H_1\text{[IV']},
\\[1.5ex]
g\times\text{[IV]}&=&
B\big[v_{r_1},\ldots,v_{r-1}\big]=B\text{[IV]},
\\[1.5ex]
g\times\text{[V]}&=&
A\big[u_{r},\ldots,u_{r+r_2-1}\big]
+H_0\big[v_{r},\ldots,v_{r+r_2-1}\big]
\\[.5ex] && 
+H_1\big[v_{r-1}+(-1)^{n-r-1}s_{n-r}\cdot v_{n-1},\ldots,
v_{r+r_2-2}+(-1)^{n-r-r_2}s_{n-r-r_2+1}\cdot v_{n-1}\big]
\\[.5ex]&=&
A\text{[V]}+H_0\text{[VI]}+H_1\text{[VI']}.
\end{array}
$$
The slice to the $\Aut(\mbb L)$-action is obtained in several steps. 
(Recall that $\bsymb{x}, \bsymb{e}$ are generic.) 

-- By using the $\Gl(r_2)$-action, we may assume that [IV]$=\bone_{r_2}$.

-- We cancel [III]: just take 
$g\,{=}\left[
\begin{array}{c|c}
\bone_{r_1}&H(z)=-\text{[III]}\cdot\text{[IV]}^{-1}
\\ \hline 
0&\bone_{r_2}
\end{array}\,
\right]\!.$

-- Also cancel [V] with an appropriate 
$g\,{=}\left[
\begin{array}{c|c}
\bone_{r_1}&H(z)=H_0+zH_1\\ \hline 0&\bone
\end{array}\,
\right]\!$, while keeping [III]$=0$ and [IV]=$\bone_{r_2}$. Indeed, the equation 
$g\times\text{[III]}=0$ yields $H_0=-H_1\text{[IV'][IV]}^{-1}$, and then 
$g\times\text{[V]}=0$ has a unique solution 
$H_1=\text{[V]}\cdot\big(\text{[IV'][IV]}^{-1}\text{[VI]}-\text{[VI']}\big)^{-1}$.  

-- Finally, by using 
$g\,{=}\left[
\begin{array}{c|c}A&0\\ \hline 0&\bone\end{array}\,
\right]\!,$ 
we may assume [I]=$\bone_{r_1}$, while keeping [III]=[V]=$0$, [IV]=$\bone_{r_2}$. 

\nit Overall, by using the $\Aut(\mbb L)$-action, the generic $\bsymb{e}$ \eqref{eq:xi} 
can be brought into the form 
\begin{equation}\label{eq:xii}
\left[
\begin{array}{c|c|c|c}
\bone_{r_1}&0&0&\cdots
\\ 
\hline 
*&\bone_{r_2}&*&\cdots
\end{array}
\right].
\end{equation}
\nit\textit{Claim}\quad Suppose $\bsymb{x}$ and $\bsymb{e}$ are generic. 
More precisely, the following matrices should be invertible:  
\begin{center}
$\text{[I]},\kern1ex\text{[IV]},\kern1ex
\text{[IV']}\text{[IV]}^{-1}\text{[VI]}-\text{[VI']}.$\footnote{These conditions 
are indeed generic: take \textit{e.g.} [I]$=\bone_{r_1}$, [IV]$=\bone_{r_2}$, 
[VI]$=$diagonal matrix, and $v_{n-1}=0$.}
\end{center}
If both $\bsymb{e},\,g\times\bsymb{e}$ are of the form \eqref{eq:xii}, then $g=\bone$. 
\\ \centerline{ 
$\begin{array}{clccl}
g\times\text{[IV]}=\bone_{r_2}&\Rightarrow B=\bone_{r_2},
&& 
g\times\text{[III]}=0,\;g\times\text{[V]}=0&\Rightarrow H_0=H_1=0, 
\\ 
g\times\text{[I]}=\bone_{r_1}&\Rightarrow A=\bone_{r_1}.&&&
\end{array}$
} 
Henceforth, we shrink $\cal H$ to the open subset appearing in the claim above. 


\subsubsection{The $(\zeta_*\cO_{\cal Z})^\times$-action}\label{sssct:T} 
Now we turn our attention to the second action. 
Unfortunately $T:=(\zeta_*\cO_{\cal Z})^\times$ is not a group, but rather a group 
scheme over $\cal H$ with fibres isomorphic to $(\mbb C^*)^n$. 
(Notice that $\mbb C^*$ is still diagonally embedded in $(\zeta_*\cO_{\cal Z})^\times$.) 
We denote by 
\\ \centerline{
$F:=\Spec_{\Hilb^n_{\mbb P^1}}(\Sym^\bullet\zeta_*\cO_{\cal Z})
\cong\cal H\times\mbb A^{n}$ 
} 
the linear fibre space determined by $\zeta_*\cO_{\cal Z}$.
Then $T$ acts diagonally on $F^2:=F\times_{\cal H}F$, 
and the action on $E$ consists in repeating $r$ times the action on $F^2$. 

The $T$-action on $F$ is complicated in the trivialization \eqref{eq:E-triv}, 
but is easy to understand the $q^*T:=(\xi_*\cO_{\cal X})^\times$-action on 
$q^* F=\cal U\times\mbb A^{n}$ in a different trivialization. Indeed, 
\\ \centerline{
$\begin{array}{l}
\kern4ex(\xi_*\cO_{\cal X})^\times\cong(\mbb C^*)^n,
\\[1ex] \disp 
\frac{\mbb C[x_1,\ldots,x_n][z]}{\langle (z-x_1)\cdot\ldots\cdot(z-x_n)\rangle}
\cong\mbb C[x_1,\ldots,x_n]^{\oplus n}
\kern-3pt=
\pr_{\mbb P^1}^*\big(\cO_{\mbb P^1,x_1}\oplus\ldots\oplus\cO_{\mbb P^1,x_n}\big)
\;\text{is a ring isomorphism,}
\end{array}$
}
by the Chinese remainder theorem, and $(t_1,\ldots,t_n)\in(\mbb C^*)^n$ acts on the 
$j$-th coordinate of $\mbb A^n$ by $t_j$. (The difficulty with the $T$-action on 
$\zeta_*\cO_{\cal Z}$ is that $x_1,\ldots,x_n$ are permuted by $\mfrak S_n$.) 

We consider the $\mfrak S_n$-invariant linear subspace $S''\subset\mbb A^{rn}_\rgt$ 
consisting of matrices of the form \eqref{eq:xi} with the top line filled by 
an arbitrary $c\in\mbb C$, that is 
\begin{equation}\label{eq:xiii}
\left[
\begin{array}{c|c|c}c&\ldots&c\\ *&\ldots&*\end{array}
\right],
\end{equation}
and define 
$\tld\Xi'':=\cal U\times\mbb A^{rn}_\lft\times S''\subset q^*E=\cal U\times_{\cal H}E$. 
Clearly, the generic element of $q^*E$ can be brought into such a form by using $q^*T$, 
uniquely up to the diagonal $\mbb C^*$-action. 
(Thus, $\tld\Xi''$ is a slice for a `$(\mbb C^*)^n/\mbb C^*_{\rm diag}$-action'.) 

The next step is to descend $\tld\Xi''$ to $E$ itself. This is \emph{not immediate}, 
since afterward we wish to take the slice \eqref{eq:xii} for the 
$(\mfrak S_n\times\Aut(\mbb L))$-action, but $\tld\Xi''$ is not $\Aut(\mbb L)$-invariant. 
Fortunately, the no-name lemma \cite{re,dom} comes to the rescue. 
We consider the diagram 
\\ \centerline{
$\xymatrix@R=1.5em@C=5em{
\cal U\times\mbb A^{rn}_\lft\times\mbb A^{rn}_\rgt
\ar@{-->}[r]^(.5){({\rm id},\Upsilon)}\ar[d]_-{\pr}
&
\cal U\times\mbb A^{rn}_\lft\times\mbb A^{rn}_{\emptyset}
\ar[d]^-{\pr}
\\ 
\cal U\times\mbb A^{rn}_\lft
\ar@{=}[r]\ar[d]
&
\cal U\times\mbb A^{rn}_\lft
\ar[d]
\\ 
\cal U
\ar@{=}[r]
&
\cal U
}$
}\\[.5ex]
By our discussion at \ref{sssct:L}, the generic $(\mfrak S_n\times\Aut(\mbb L))$-stabilizer 
on $\cal U\times\mbb A^{rn}_\lft$ is trivial. (Notice that $\mfrak S_n$ acts trivially on 
$\mbb A^{rn}_\lft$ because the trivilization \eqref{eq:E-triv} holds on $\cal H$.) 
Then there is a $(\mfrak S_n\times\Aut(\mbb L))$-invariant open subset 
$\tld O\subset\cal U\times\mbb A^{rn}_\lft$, and a birational $\pr$-fibrewise linear 
map $({\rm id},\Upsilon)$ which is equivariant for the following actions: 

-- $\mfrak S_n$ acts on $\mbb A^{rn}_\rgt$ by permuting the 
$n$ copies of $\mbb A^{r}$;

-- $\Aut(\mbb L)$ acts on $\mbb A^{rn}_\rgt$ the same as on $\mbb A^{rn}_\lft$. 
(Anyway, $\Aut(\mbb L)$ acts diagonally on $\mbb A^{2rn}$.);

-- $\mfrak S_n\times\Aut(\mbb L)$ acts trivially on $\mbb A^{rn}_{\emptyset}$.

\nit The group $q^*T$ acts both on the fibre and the base of $\pr$; it is a priori unclear 
whether $\tld O$ is $q^*T$-invariant. (Although is likely that is possible to arrange this.) 
However, the $q^*T$-orbit of the generic point in 
$\cal U\times\mbb A^{rn}_\lft\times\mbb A^{rn}_\rgt$ intersects 
$\tld\Xi''|_{\tld O}=\tld O\times S''$ along a unique $\mbb C^*_{\rm diag}$-orbit 
(a straight line). 
Indeed, the generic stabilizer is trivial, and the dimension of the $q^*T$-orbit of the 
`bad locus' $(\cal U\times\mbb A^{rn}_\lft\sm\tld O)\times S''$ 
is at most  $\dim(\cal U\times\mbb A^{rn}_\lft\sm\tld O)+rn<\dim\cal U+2rn$. 
We consider the `$q^*T/\mbb C^*_{\rm diag}$-slice' 
\\ \centerline{
$\tld\Xi''_\emptyset:=({\rm id},\Upsilon)(\tld\Xi''|_{\tld O})
\subset
\tld O\times\mbb A^{rn}_\emptyset.$
}
By the $\pr$-linearity of $({\rm id},\Upsilon)$, $\tld\Xi''_\emptyset$ is still a linear 
fibration over $\tld O$, invariant under the $\mfrak S_n$-action (as $S''$ is so). 
At this stage only, we take the quotient by $\mfrak S_n$, and get the 
`$T/\mbb C^*_{\rm diag}$-slice' 
\\ \centerline{
$\Xi''_\emptyset=\tld\Xi''_\emptyset/\mfrak S_n
\subset
\big(\tld O/\mfrak S_n\big)
\times\mbb A^{rn}_\emptyset\subset E.$
\quad 
We denote $O:=\tld O/\mfrak S_n\subset\cal H\times\mbb A^{rn}_\lft$.
}
The essential property is that 
$\Aut(\mbb L)=\big(\Aut(\mbb L)\times\mbb C^*_{\rm diag}\big)/\mbb C^*$ 
acts on $O$ as in \ref{sssct:L}, and by multiplication on the fibres of 
$\pr{:}\,O\times\mbb A^{rn}_\emptyset\to O$. 


\subsubsection{The $G$-action}\label{sssct:G} 
Now we assemble \ref{sssct:L} and \ref{sssct:T} to produce the $G$-slice on $E$. 
We define 
\begin{equation}\label{eq:slice}
S:=\Xi''_\emptyset|_{O\cap(\cal H\times S')}\subset E,\;
\text{ where $S'$ consists of matrices of the form \eqref{eq:xii}.}
\end{equation}
The intersection $O\cap(\cal H\times\Xi')$ is non-empty, because 
$\Aut(\mbb L)\cdot(\cal H\times S')$ is open. 
Clearly, $S\subset E$ is $(n+r^2-1)$-co-dimensional, it is an open subset of a locally 
trivial fibration over some open $\cal H\subset{\rm Hilb}^n_{\mbb P^1}$. Our 
discussion shows that any the $G$-orbit of the generic intersects $S$ at only one point. 
Indeed, the $T$-orbit intersects $\Xi''_\emptyset$ along a $\mbb C^*_{\rm diag}$-orbit, 
and we use the remaining $\Aut(\mbb L)$-action to move the point over 
$O\cap(\cal H\times S')$. 
\end{proof}

\begin{corollary}\label{cor:P2}
$M_{\mbb P^2}(r;0,n)$ is a rational variety, for all $n\ges r\ges 2$. 
\end{corollary}

For $r=2$, the statement in proved in \cite{mae} (see also \cite{es}). 
For arbitrary $r$ and $c_1$, the best results are obtained in \cite{c-mr,c-mr3,yo}. 

\begin{proof}
Let $\si:Y_1\to\mbb P^2$ be the contraction of $\L$ (equivalently, the blow-up 
of a point in $\mbb P^2$). According to \ref{thm:hirz} and \ref{thm:rtl}, 
$M_{Y_1}^{L_c}(r;0,n)$, where $c$ is sufficiently large, is irreducible and rational, 
and there is an open subset of subset consisting of vector bundles $\cF$ whose restrictions 
to both $\L\subset Y_1$ and the generic $\l\in|\cO_\pi(1)|$ are trivializable. For any such 
$\cF$, the direct image $\si_*\cF$ is locally free and semi-stable on $\mbb P^2$. 
This yields a rational map 
\\ \centerline{
$\si_*:\bar M_{Y_1}^{L_c}(r;0,n)\dashto\bar M_{\mbb P^2}(r;0,n),
\quad\cF\mt\si_*\cF.$
} 
It is obviously injective on the open locus formed by $\cF$ as above. 
Moreover, $\si_*$ is dominant, because $\Ext^1(\cF,\cF)\cong\Ext^1(\hat\cF,\hat\cF)$, 
for any $\cF$ which is trivializable along $\L$. 
Zariski's main theorem implies that $\si_*$ is birational. 
\end{proof}

Somewhat unexpectedly, moduli spaces of framed vector bundles will naturally occur 
in the next section. 

\begin{definition}[See \textit{e.g.} \cite{do,bbr}]\label{def:frame} 
A \emph{framing} of a sheaf $\cF$ on $Y_\ell$ along a reduced, irreducible curve 
$\l\subset Y_\ell$ is an isomorphism $\cF_\l\srel{\th}{\to}\cO_\l^{\oplus r}$. 
Two framings of $\cF$ along $\l$ are equivalent if there is an automorphism 
of $\cF$ sending one framing into the other. 
\end{definition}

Thus, if $\cF$ is stable, two framings $\th,\th'$ are equivalent if $\th'=c\th$, for 
$c\in\mbb C^*$. 
We deduce that, in this latter case, the choice of a basis in $\Gamma(\cF_\l)$ 
(modulo $\mbb C^*$) determines a framing of $\cF$ along $\l$. 
Therefore, the possible framings of $\cF$ along $\l$ is the ${\PGl}(r)$-orbit 
of a given framing. 
We are going to show how the slice $S$ constructed above allows to do this for 
families of vector bundles. 
For shorthand, we let $M_{Y_\ell}^\vb:=M_{Y_\ell}^{L_c}(r;0,n)^\vb$, 
and denote by $M_{Y_\ell,\l}^\vb$ the moduli space of $\l$-framed, $L_c$-stable 
vector bundles on $Y_\ell$. 

\begin{corollary}\label{cor:rtl-frameY}
Assume that the restriction to $\l$ of the generic $\cF\in M_{Y_\ell}^\vb$ 
is trivializable. Then $M_{Y_\ell,\l}^\vb$ is birational to 
$M_{Y_\ell}^\vb\times\PGl(r)$, 
so $M_{Y_\ell,\l}^\vb$ is an irreducible, $2nr$-dimensional, rational variety. 
\end{corollary}

\begin{proof}
Consider the sheaf $\cal E$ \eqref{eq:E}, and the corresponding linear fibre space 
$\bar\zeta:E\to\cal H$. We denote by $t$ the tautological section of $\bar\zeta^*\cal E$ 
over $E$. Also, notice that $\tld{\cal Z}:=E\times_{\cal H}\cal Z\times_{\mbb P^1}Y$ 
is a subvariety of $E\times Y$, since $\cal Z\subset\cal H\times\mbb P^1$.  
We consider the following composition of homomorphisms: 
\begin{equation}\label{eq:tau}
\xymatrix@C=3em{
{(\pi\circ\pr_Y^{E\times Y})}^*\,\mbb L^\vee\ar[r]\ar[drr]_-u
&
{(\pi\circ\pr_Y^{E\times Y})}^*\,\mbb L^\vee\otimes\cO_{\tld{\cal Z}}
\ar[r]^-{\langle\cdot,t\rangle}_-{\text{pairing}}
&
{(\pi\circ\pr_{Y}^{E\times Y})}^*\pi_*\cO_\pi(1)\otimes\cO_{\tld{\cal Z}}
\ar[d]^-{\text{evaluation}}
\\&&
{(\pr_Y^{E\times Y})}^*\cO_\pi(1)\otimes\cO_{\tld{\cal Z}}.
}
\end{equation}
At generic $\bsymb{x}\in\Hilb^n_{\mbb P^1}$, $\bsymb{e}\in\cal E_{\bsymb{x}}$, 
$u$ is 
\\ \centerline{
$\pi^*\mbb L^\vee\surj
\ouset{i=1}{n}{\bigoplus}\,\mbb L_{x_i}^\vee\srel{\langle\cdot,t\rangle}{\lar}
\ouset{i=1}{n}{\bigoplus}\Gamma(\cO_{\pi^{-1}(x_i)}(1))\surj
\ouset{i=1}{n}{\bigoplus}\,\cO_{\pi^{-1}(x_i)}(1).$
} 
The pairing with $t$ is surjective over the locus $E'$ consisting of $\bsymb{e}\in E$, 
such that $t(\bsymb{e})$ has linearly independent components at each 
$x\in\bar\zeta(\bsymb{e})$. (Obviously, $E'$ is $G$-invariant.) 
Then $\bsymb{\cF}:=\Ker(u)^\vee|_{(E'\cap S)\times Y}$ is a locally free sheaf over 
$S'\times Y:=(E'\times Y)\cap(S\times Y)$. 
(The intersection is non-empty, as $G\cdot S\subset E$ is dense.) 
Since $S$ is birational to $M_{Y_\ell}^\vb$, $\bsymb{\cF}$ is a universal sheaf 
over $M'\times Y$, for some (non-empty) open subset $M'\subset M_{Y_\ell}^\vb$. 

After possibly shrinking $M'$, we may assume that $\cF_\l\cong\cO_\l^{\oplus r}$, for 
all $\cF\in M'$. Thus  $(\pr_{M'}^{M'\!\times\,\l})_*\,\bsymb{\cF}$ is locally free over 
$M'$, and, after shrinking $M'$ further, we may assume that this latter is trivializable 
over $M'$. Then the moduli space of vector bundles $\cF\in M'$ which are $\l$-framed 
is isomorphic to $M'\times\PGl(r)$. 
\end{proof}

\begin{remark}\label{rmk:2frame}
In \ref{thm:MM}, we will encounter two types of framings: 

\nit{\rm(i)} 
$\l=\pi^{-1}(o)$ is a fibre of $\pi$. Any $\cF\in M_{Y_\ell,\l}^\vb$ 
fits into $0\,{\to}\,\pi^*\pi_*\cF\srel{f}{\to}\cF{\to}\,\cal Q_\cF\,{\to}\,0$, 
so $\cF_\l\cong\cO_\l^{\oplus r}$, if $\Supp(\cal Q_\cF)$ is disjoint of $\l$. 

\nit{\rm(ii)} 
The second kind of framings is along a general $\l\in|\cO_\pi(1)|$. 
(In \ref{thm:MM}, we will need $\ell=1$.) 
We already observed in \ref{thm:hirz}(i) that such a $\l$-framed vector bundle 
is automatically $L_c$-semi-stable. 
\end{remark}

As an immediate application, we deduce the following (apparently new) statement.

\begin{corollary}\label{cor:rtl-Y}
The moduli spaces of framed sheaves with $c_1=0$ constructed in \cite{bbr} are rational.
\end{corollary}

\begin{proof}
Use \ref{cor:rtl-frameY} and \ref{rmk:2frame}.
\end{proof}


\section{Application: %
stable vector bundles on $\mbb P^2$-bundles over $\mbb P^1$}{\label{sect:p21}}

Stable vector bundles on $\mbb P^2$ are studied in \cite{bh,hu,oss}. Here we give 
a monad theoretic construction of stable vector bundles $\cF$ with $c_1=0$ on 
$\mbb P^2$-fibre bundles over $\mbb P^1$. Let us remark that \cite{c-mr2} studies the 
moduli space of \emph{rank two} vector bundles on projective bundles over curves. 
If a vector bundle obtained this way has $c_1=0$, its $c_2$ is necessarily of the form 
$-(k^2u^2+2kluv)$, with $u,v$ as in \eqref{eq:h2y} below and $k,l\in\mbb Z$. Thus 
our work brings novelties even in the rank two case.

For two integers $0\les a\les b$, 
we consider 
$Y=Y_{a,b}
:=\mbb P(\cO_{\mbb P^1}\oplus\cO_{\mbb P^1}(-a)\oplus\cO_{\mbb P^1}(-b))$. 
The $3$-fold $Y$ admits the projection $\pi:Y\to\mbb P^1$ with fibres isomorphic 
to $\mbb P^2$. (We say that $Y$ is a $\mbb P^2$-fibre bundle over $\mbb P^1$.) 
Let $\fbr=\mbb P^2$ be the generic fibre of $\pi$. The relatively ample line bundle 
$\cO_\pi(1)$ on $Y$ is big, globally generated, $\pi$-ample (except for $a=b=0$, 
when it satisfies \ref{rmk:weak}), and holds:
\begin{equation}\label{eq:h2y}
\begin{array}{l}
H^2(Y;\mbb Z)=\mbb Z\cdot
\underbrace{[\cO_{\pi}(1)]}_{=:\,u}
\;\oplus\;
\mbb Z\cdot
\underbrace{[\cO_{\mbb P^1}(1)]}_{=:\,v}\,,
\quad\text{with}\quad u^3=(a+b)\cdot u^2v=a+b\,;
\\[1.5ex]
\kappa_\pi=-3u+(a+b)v\quad\text{and}\quad
\kappa_Y=-3u+(a+b-2)v.
\end{array}
\end{equation}
Here $\kappa_\pi$ and $\kappa_Y$ stand for the relative and the (absolute) 
canonical class of $Y$ respectively. The `exceptional line' 
$\Lambda:=\mbb P(\cO_{\mbb P^1}\oplus 0\oplus 0)\subset Y$ has the property 
that $\cO_\pi(1)\otimes\cO_\Lambda\cong\cO_\Lambda$. 

\subsection{Review of the monad construction on $\mbb P^2$} 
In \cite[section 6]{bh} is proved that any stable vector bundle $\cal V$ 
on $\mbb P^2$ is the cohomology of a certain monad on $\mbb P^2$. 
For completeness, we briefly recall some details. 
For a {\em semi-stable}\footnote{In \cite[Section 6]{bh} the authors assume that $\cal V$ 
is {\em stable}. However, one can easily check that the statements below are valid for 
$\cal V$ semi-stable. The reason is that, in {\it loc. cit.}, the authors consider also minimal 
$-2$-resolutions, which indeed require the stability of $\cal V$.}, rank $r$ vector bundle 
$\cal V$ on $\Phi$, with $c_1(\cal V)=0$ and $c_2(\cal V)=n$, the following hold: 
\begin{enumerate}
\item 
$\left\{\begin{array}{l}
\Gamma(\fbr,\cal V\otimes\cO_{\fbr}(-j))
=H^2(\fbr,\cal V\otimes\cO_{\fbr}(-j))=0
\\[1ex] 
\text{and}\quad\dim H^1(\fbr,\cal V\otimes\cO_{\fbr}(-j))=n, 
\end{array}\right.$ 
for $j=1,2$.\vskip1ex 

\item 
$\underset{l\ges 0}{\mbox{$\bigoplus$}}H^1\bigl(\fbr,\cal V\otimes\cO_\fbr(l-1)\bigr)$ 
is generated by $H^1(\fbr,\cal V\otimes\cO_\fbr(-1))$ over 
$\underset{l\ges 0}{\mbox{$\bigoplus$}}\Gamma\bigl(\fbr,\cO_\fbr(l)\bigr)$. 

\item The identity in 
$\End\bigl(\,H^1(\fbr,\cal V\otimes\mbb \cO_{\fbr}(-1))\,\bigr){\cong}
\Ext^1\bigl(\,H^1(\fbr,\cal V\otimes\cO_{\fbr}(-1))\otimes\cO_{\fbr}(1),\cal V\,\bigr),$ 
defines the {\em minimal $-1$-resolution} of $\cal V$ \cite[Section 2]{bh}: 
\begin{equation}\label{eq:Q}
0\to\cal V\to\cal Q_\Phi\to 
H^1(\fbr,\cal V\otimes\mbb \cO_{\fbr}(-1))\otimes\cO_{\fbr}(1)\to 0. 
\end{equation}
Similarly, we consider the minimal $-1$-resolution of $\cal V^\vee$,  and obtain 
the display of a (minimal) monad whose cohomology is $\cal V$: 
\begin{equation}\label{eq:O}
\scalebox{.8}{
\xymatrix@C=1.5em@R=1.1em{
H^1(\fbr,\cal V^\vee(-1))^\vee\otimes\cO_{\fbr}(-1)\;
\ar@{^(->}[r]\ar@{=}[d]&
\kern1ex\cal K_\fbr\ar@{->>}[r]\ar@{^(->}[d]&
\cal V\ar@{^(->}[d]
\\ 
\mbox{$
\underbrace{H^1(\fbr,\cal V^\vee(-1))^\vee\otimes\cO_{\fbr}(-1)}_{=:\;\cal A_\fbr}\;
$}
\ar@{^(->}[r]^-{{\rm\bf A}_\fbr}&
\cO_\Phi^{\oplus r+2n}
\ar@{->>}[r]\ar@{->>}[d]^-{{\rm\bf B}_\fbr}
&
\cal Q_\fbr\ar@{->>}[d] 
\\ 
&
\mbox{$
\underbrace{H^1(\fbr,\cal V(-1))\otimes\cO_{\fbr}(1)}_{=:\;\cal B_\fbr}
$}
\ar@{=}[r]
&
H^1(\fbr,\cal V(-1))\otimes\cO_{\fbr}(1)
}}
\end{equation}

\item If $\cal V$ on $\fbr$ is stable, then $h^1(\cal V)=n-r$, thus $n\ges r$.

\item The vector bundles whose restriction to a (straight) line $\l\subset\fbr$ 
is isomorphic to $\cO_\l^{\oplus r}$ form an open and dense subset of the 
moduli space of semi-stable vector bundles on $\fbr$, with $c_1=0,c_2=n$. 
(See \cite[Lemma 2.4.1]{hu} for the proof.)
\end{enumerate}

\subsection{The relative monad construction on $Y_{a,b}$} 
Our main result is the following:

\begin{theorem}\label{thm:p2-p1}
Let $\cF$ be a rank $r$ vector bundle on $Y$ whose Chern classes are 
\begin{equation}\label{eq:chern}
c_1(\cF)=0,\;c_2(\cF)=n\cdot u^2,\;c_3(\cF)=0,
\end{equation}
and which is semi-stable with respect to $L_c:=\cO_\pi(1)+c\cO_{\mbb P^1}(1)$, 
for $c> r(r-1)n(a+b)$. 

We assume that $\cF$ has the following properties: 
\begin{eqnarray}
\text{\rm(i)}&
\begin{minipage}[t]{0.8\textwidth}
The restriction of $\cF$ to a line $\l\cong\mbb P^1$ in the generic fibre 
$\fbr\cong\mbb P^2$ is trivial.
\end{minipage}\label{eq:restr}
\\ & 
\begin{minipage}[t]{0.8\textwidth}
The restriction $\cF_\Lambda:=\cF\otimes\cO_\Lambda$ to the exceptional 
line $\Lambda$ is trivial.
\end{minipage}
\label{eq:restrL}
\\[0ex] 
\text{\rm(ii)}&
\begin{minipage}[t]{0.8\textwidth}
$R^2\pi_*\bigl(\cF\otimes\cO_\pi(-2)\bigr)=0.$
\end{minipage}\label{eq:R2}
\\[0ex] 
&
\begin{minipage}[t]{0.8\textwidth}
$H^1\bigl(
Y,\cF\otimes\cO_\pi(-1)\otimes\pi^*\cO_{\mbb P^1}(-1)
\bigr)=0,$ 
\end{minipage}\label{eq:F}
\\ 
&
\begin{minipage}[t]{0.8\textwidth}
$H^1\bigl(
Y,\cF\otimes\cO_\pi(-2)\otimes\pi^*\cO_{\mbb P^1}(a+b-1)
\bigr)=0.$
\end{minipage}\label{eq:Fv}
\end{eqnarray}
Then $\cF$ can be written as the cohomology of a monad of the form 
\begin{equation}\label{eq:AOB}
\cO_\pi(-1)^{\oplus n}
\srel{\rm\bf A}{\lar}
\cO_Y^{\oplus r+2n}
\srel{\rm\bf B}{\lar}
\cO_\pi(1)^{\oplus n}.
\end{equation}
If $\cF_\fbr$ is stable, then this monad is uniquely defined, up to the action of 
\begin{equation}\label{eq:G}
G:=\Gl(n){\times}\Gl(r+2n){\times}\Gl(n):\quad 
(g_1,g,g_2)\times({\bf A},{\bf B}):=(g{\bf A}g_1^{-1},g_2{\bf B}g^{-1}).
\end{equation} 
The isotropy group of this action is $\mbb C^*$, diagonally embedded in $G$. 
\end{theorem}

\begin{remark}\label{rmk:hyp} 
Before starting the proof, we analyse the various conditions imposed on $\cF$.  

-- $\cO_\pi(\pm1)$ is trivial along $\L$, thus \eqref{eq:restrL} is necessary. 
Also, a simple diagram chasing in the display of \eqref{eq:AOB} yields  
\eqref{eq:R2}, \eqref{eq:F}, \eqref{eq:Fv}. 
Hence \eqref{eq:restrL} -- \eqref{eq:Fv} \emph{must be imposed}. 

-- \eqref{eq:R2} should be interpreted as a weak, $\pi$-relative semi-stability condition 
for $\cF$, because $R^2\pi_*\bigl(\cF\otimes\cO_\pi(-2)\bigr)$ is a torsion sheaf on 
$\mbb P^1$, anyway. Moreover, if $\cF$ is semi-stable on each fibre of $\pi$, 
then \eqref{eq:R2} is automatically satisfied. However, as we explained in the 
introduction, we avoid this requirement in order to enlarge the frame of \cite{ma0,ma,si}.

-- \eqref{eq:restr} is the only assumption imposed by technical reasons. 
(Is needed to control the middle term of the monad \eqref{eq:AOB}.) 
It should be viewed as a genericity condition for $\cF$. 
Indeed, the restriction to $\fbr$ of \textit{any} $L_c$-semi-stable vector bundle on $Y$ 
is $\cO_\fbr(1)$-semi-stable; our previous discussion (point (v) above) states that 
most semi-stable vector bundles on $\fbr$ are trivializable along $\l$.  
For $r=2$, the Grauert-M\"ullich theorem (see \cite[Chapter 3]{hl}) implies that 
\eqref{eq:restr} is automatically satisfied. 
\end{remark}

Throughout this section, for $x\in\mbb P^1$, we denote 
$\phi_x:=\pi^{-1}(x)\cong\mbb P^2$. When (hopefully) no confusion is possible, 
we write $\cF(-1):=\cF\otimes\cO_\pi(-1)$, 
and similarly for $\cF^\vee,\cal Q,\cal M$, \textit{etc}. 
First we clarify the rationale for \eqref{eq:F} and \eqref{eq:Fv}.

\begin{lemma}\label{lm:coh}
Assume that $\cF$ on $Y$ is semi-stable and satisfies \eqref{eq:chern}, \eqref{eq:R2}. 
Then hold:
\begin{enumerate}
\item For $l=-2,-1$, and $k\ges -2$ we have 
$\begin{array}[t]{cc}
\pi_*\bigl(\cF\otimes\cO_\pi(l)\bigr)=0,
&
R^2\pi_*\bigl(\cF\otimes\cO_\pi(k)\bigr)=0;
\\[1ex] 
\pi_*\bigl(\cF^\vee\otimes\cO_\pi(l)\bigr)=0,
&
R^2\pi_*\bigl(\cF^\vee\otimes\cO_\pi(k)\bigr)=0;
\end{array}$  

\item $R^1\pi_*\bigl(\cF\!\otimes\cO_\pi(-2)\bigr)$, 
$R^1\pi_*\bigl(\cF^\vee\!\otimes\cO_\pi(-2)\bigr)$ 
are locally free of rank $n$ and degree $-n(a+b)$; 

\nit $R^1\pi_*\bigl(\cF\!\otimes\cO_\pi(-1)\bigr)$, 
$R^1\pi_*\bigl(\cF^\vee\!\otimes\cO_\pi(-1)\bigr)$  
are locally free of rank $n$ and degree zero. 

\item Moreover, the following implications hold: 
$$
\begin{array}{l}
\eqref{eq:F}\quad\Longrightarrow\quad 
R^1\pi_*\bigl(\cF\otimes\cO_\pi(-1)\bigr)\cong\cO_{\mbb P^1}^{\oplus n},
\quad
R^1\pi_*\bigl(\cF^\vee\otimes\cO_\pi(-2)\bigr)\cong\cO_{\mbb P^1}(-a-b)^{\oplus n}
\\[1.5ex]
\eqref{eq:Fv}\quad\Longrightarrow\quad 
R^1\pi_*\bigl(\cF^\vee\otimes\cO_\pi(-1)\bigr)\cong\cO_{\mbb P^1}^{\oplus n},
\quad
R^1\pi_*\bigl(\cF\otimes\cO_\pi(-2)\bigr)\cong\cO_{\mbb P^1}(-a-b)^{\oplus n}.
\end{array}
$$

\item The natural homomorphism 
$
H^1\bigl(Y,\cF\!\otimes\!\cO_\pi(-1)\bigr)\otimes\pi_*\cO_\pi(l+1)
\!\to\! 
R^1\pi_*\bigl(\cF\!\otimes\!\cO_\pi(l)\bigr)
$ 
is surjective, for all $l\ges -1$. 
\end{enumerate}
\end{lemma}

\begin{proof} (i) 
The restrictions of $\cF,\cF^\vee$ to the generic fibre $\fbr$ are semi-stable of degree 
zero, so $\pi_*\bigl(\cF(l)\bigr)=\pi_*\bigl(\cF^\vee(l)\bigr)=0$, for $l=-2,-1$, 
because they are both torsion free sheaves. 

A generic divisor $D\in|\cO_\pi(1)|$ is isomorphic to the Hirzebruch surface 
$\mbb P\bigl(\cO_{\mbb P^1}(-a)\oplus\cO_{\mbb P^1}(-b)\bigr)$. 
The push-forward by $\pi$ of $0\to\cF(-2)\to\cF(-1)\to\cF_D(-1)\to 0$ yields 
\\ \centerline{
$\underbrace{R^2\pi_*\bigl(\cF(-2)\bigr)}_{=0}\to 
R^2\pi_*\bigl(\cF(-1)\bigr)\to
\underbrace{R^2\pi_*\bigl(\cF_D(-1)\bigr)}_{=0}$. 
}
The same argument shows $R^2\pi_*\bigl(\cF(k)\bigr)=0$, for all $k\ges -2$. 

By repeatedly applying the semi-continuity theorem \cite[Ch. III, Theorem 12.11]{ha}, 
we are going to prove that $R^2\pi_*\bigl(\cF^\vee(l)\bigr)=0$, for $l=-2,-1$. 
For shorthand, let $\cG:=\cF\otimes\cO_\pi(l)$. 
$$\kern-2ex
\begin{array}{rl}
\biggl.
\begin{array}{l}
(R^2\pi_*\cG)_x
\!\to\! 
H^2(\phi_x,\cG_x)\text{ is surjective because }\dim_{\mbb P^1}Y=2,
\\[1ex]
R^2\pi_*\cG=0\text{ is locally free} 
\end{array}
\kern-1ex\biggr\}
&
\kern-.5ex\Rightarrow\kern-.5ex
\begin{array}{l}
(R^1\pi_*\cG)_x
\to H^1(\phi_x,\cG_x)
\\ 
\text{is isomorphism, }\forall\,x\in\mbb P^1.
\end{array}
\\[3ex] 
\begin{array}{l}
(R^1\pi_*\cG)_x
\to H^1(\phi_x,\cG_x)\text{ is surjective},
\\[1ex]
R^1\pi_*\cG\text{ is locally free because }\pi_*\cG=R^2\pi_*\cG=0 
\end{array}
\kern-1ex\biggr\}
&
\kern-.5ex\Rightarrow\kern-.5ex
\begin{array}{l}
(\pi_*\cG)_x
\to\Gamma(\phi_x,\cG_x)
\\ 
\text{is isomorphism, }\forall\,x\in\mbb P^1.
\end{array}
\end{array}
$$
For $l=-2,-1$, we deduce that 
$
H^2(\phi_x,\cF^\vee(l))\cong\Gamma(\phi_x,\cF(\underbrace{-3-l}_{=-2,-1}))=0,
\;\forall\,x\in\mbb P^1.
$ 
Grauert's criterion \cite[Ch. III, Corollary 12.9]{ha} implies now that 
$R^2\pi_*\bigl(\cF^\vee(l)\bigr)=0$.\smallskip  


\nit(ii) 
Since $\pi_*\bigl(\cF(l)\bigr), R^2\pi_*\bigl(\cF(l)\bigr)$ and 
$\pi_*\bigl(\cF^\vee(l)\bigr), R^2\pi_*\bigl(\cF^\vee(l)\bigr)$ vanish for $l=-2,-1$, it 
follows that $R^1\pi_*\bigl(\cF(l)\bigr), R^1\pi_*\bigl(\cF^\vee(l)\bigr)$ are locally free. 
Their rank and degree are given by the Grothendieck-Riemann-Roch formula. 
\smallskip 

\nit(iii) The assumption \eqref{eq:F} implies \\ 
\centerline{$
0=H^1\bigl(
Y,\cF\otimes\cO_\pi(-1)\otimes\pi^*\cO_{\mbb P^1}(-1)
\bigr)
=
\Gamma\bigl(\,
\mbb P^1,\cO_{\mbb P^1}(-1)\otimes R^1\pi_*\bigl(\cF(-1)\bigr)
\,\bigr),
$} 
so the locally free sheaf $R^1\pi_*\bigl(\cF(-1)\bigr)$ decomposes into a direct sum 
of line bundles $\cO_{\mbb P^1}(l)$, with $l\les0$. But their degrees add up to zero, 
so $R^1\pi_*\bigl(\cF(-1)\bigr)\cong\cO_{\mbb P^1}^{\oplus n}$. 
Using Serre duality, we find 
\\ \centerline{
$0=H^2\bigl(
Y, \cF^\vee(-2)\otimes\cO_{\mbb P^1}(a+b-1)
\bigr)
=
H^1\bigl(\,
\mbb P^1, 
\cO_{\mbb P^1}(a+b-1)\otimes R^1\pi_*(\cF^\vee(-2))
\,\bigr).$
} 
A similar argument as before implies 
$R^1\pi^*(\cF^\vee(-2))\cong\cO_{\mbb P^1}(-a-b)^{\oplus n}$.
\smallskip 


\nit(iv)
Let $L\cong\mbb P^1$ be the intersection of two general divisors in 
$|\cO_\pi(1)|$. The push-forward of the sequence 
$0\to\cO_\pi(-2)\to\cO_\pi(-1)^{\oplus 2}\to\cO_Y\to\cO_L\to0$ 
tensored by $\cF(l)$, with $l\ges 0$, yields 
$R^1\pi_*\bigl(\cF(l-1)\bigr)\otimes\Gamma(Y,\cO_\pi(1))
\to
R^1\pi_*\bigl(\cF(l)\bigr)$ 
is surjective. The claim follows because $\Sym^{l+1}\Gamma(Y,\cO_\pi(1))$ 
generates $\pi_*\cO_\pi(l+1)$. 
\end{proof}

Now we consider the extensions 
\begin{eqnarray}
\label{eq:extq} 
0\to\cF\to\cal Q\to 
H^1(Y,\cF\otimes\cO_\pi(-1))\otimes\cO_\pi(1)\to 0, 
\\ 
\label{eq:extk} 
0\to H^1(Y,\cF^\vee\otimes\cO_\pi(-1))^\vee\otimes\cO_\pi(-1)
\to\cal K\to\cF\to 0
\end{eqnarray}
corresponding respectively to the identity elements  
\begin{equation}\label{eq:bone}
\begin{array}{r}
\bone\in\End\bigl(\,H^1(Y,\cF\otimes\cO_\pi(-1))\,\bigr)\cong 
\Ext^1\bigl(\,
H^1(Y,\cF\otimes\cO_\pi(-1))\otimes\cO_\pi(1),\cF
\,\bigr),
\\ 
\bone\in\End\bigl(\,H^1(Y,\cF^\vee\otimes\cO_\pi(-1))\,\bigr)
\cong
\Ext^1\bigl(\,
\cF,
H^1(Y,\cF^\vee\otimes\cO_\pi(-1))^\vee\otimes\cO_\pi(-1)
\,\bigr).
\end{array}
\end{equation}
(We remark that these extensions exist for any vector bundle $\cF$.) 

\begin{lemma}
The extensions \eqref{eq:extq} and \eqref{eq:extk} can be uniquely completed 
to a monad over $Y$ whose cohomology is $\cF$, and whose restriction to the 
generic fibre of $\pi$ is the monad \eqref{eq:O}:
\begin{equation}\label{eq:M}
\scalebox{.8}{
\xymatrix@C=1.5em@R-=1.1em{
H^1(\cF^\vee(-1))^\vee\otimes\cO_\pi(-1)\;
\ar@{^(->}[r]\ar@{=}[d]&
\cal K_{\phantom{f}}\ar@{->>}[r]\ar@<-3pt>@{^(->}[d]&
\cF\ar@{^(->}[d]
\\ 
\mbox{$
\underbrace{H^1(\cF^\vee(-1))^\vee\otimes\cO_\pi(-1)}_{=:\;\cal A}\;
$}
\ar@{^(->}[r]^-{{\rm\bf A}}&
\cal M\ar@{->>}[r]\ar@{->>}[d]^-{{\rm\bf B}}&
\cal Q\ar@{->>}[d]
\\ 
&
\mbox{$
\underbrace{H^1(\cF(-1))\otimes\cO_\pi(1)}_{=:\;\cal B}
$}
\ar@{=}[r]
&H^1(\cF(-1))\otimes\cO_\pi(1)
}}
\end{equation}
\end{lemma}

\begin{proof}
Indeed, the middle entry $\cal M$ exists and is uniquely defined, because 
the left- and rightmost entries of the following exact sequence vanish: \\ 
\centerline{$
\underbrace{\Ext^1(\cal B,\cal A)}_{=\,0}
\to 
\Ext^1(\cal Q,\cal A)
\to
\Ext^1(\cF,\cal A)
\to
\underbrace{\Ext^2(\cal B,\cal A)}_{=\,0}.
$} 
The restriction of \eqref{eq:M} to the generic fibre $\fbr$ is the monad \eqref{eq:O}, 
because \eqref{eq:extq} and \eqref{eq:extk} restrict to the corresponding extensions 
\eqref{eq:Q} for $\cal V:=\cF\otimes\cO_\fbr$ and $\cal V^\vee$ respectively. 
\end{proof}

Next we study the cohomological properties of the vector bundle $\cal M$ appearing 
in \eqref{eq:M}. 

\begin{lemma}\label{lm:M1}
Let $\cF$ be a vector bundle on $Y$ satisfying \eqref{eq:restr} -- \eqref{eq:Fv}. 
Then holds: 
\\ \centerline{
$R^1\pi_*\bigl(\cal M\otimes\cO_\pi(l)\bigr)=0,\;\forall\,l\in\mbb Z$.
}
\end{lemma}

\begin{proof}
The last column of \eqref{eq:M} and lemma \ref{lm:coh}(iv) imply 
$R^1\pi_*\bigl(\cal Q(l)\bigr)=0$, $\forall\, l\ges -1$. The middle 
horizontal sequence in \eqref{eq:M} immediately yields the conclusion for $l\ges -1$. 

The case $l\les -2$ is treated in several steps. We consider $D\in|\cO_{\pi}(1)|$ generic.

\nit\textit{Step 1}\quad 
For $l\les-1$, the upper horizontal sequence yields 
$\pi_*(\cal K_D(l))\srel{\cong}{\to}\pi_*(\cF_D(l)).$ 
By \eqref{eq:restr}, $\cF_D$ is trivial along the general fibre of $D\,{\to}\,\mbb P^1$, 
so $\pi_*(\cal K_D(l))\cong\pi_*(\cF_D(l)){=}\,0$.

\nit\textit{Step 2}\quad 
For $l\les-2$, the middle vertical exact sequence yields 
$\pi_*(\cal K_D(l))\srel{\cong}{\to}
\pi_*(\cal M_D(l)),$ 
so $\pi_*(\cal M_D(l))=0$. 

\nit\textit{Step 3}\quad
After twisting $0\to\cO_\pi(-1)\to\cO_Y\to\cO_D\to0$ by $\cal M(l)$, with $l\les-2$, 
we deduce 
$0=\pi_*(\cal M_D(l))\to R^1\pi_*(\cal M(l-1))\to R^1\pi_*(\cal M(l)),\;\forall\,l\les-2.$ 
Hence it suffices to prove the vanishing of $R^1\pi_*(\cal M(-2))$. 

\nit\textit{Step 4}\quad
The upper sequence yields 
\\ \centerline{
$0\to R^1\pi_*(\cal K(-2))\to 
R^1\pi_*(\cF(-2))\to 
H^1(Y,\cF^\vee(-1))^\vee\otimes R^2\pi_*\cO_\pi(-3).$ 
} 
The rightmost arrow is injective, because the dual homomorphism \\ 
\centerline{ 
$H^1(Y,\cF^\vee(-1))\otimes\cO_{\mbb P^1}(a+b)
\to
R^1\pi_*(\cF^\vee(-1))\otimes\cO_{\mbb P^1}(a+b)$
}
is surjective by lemma \ref{lm:coh}(iv). 
(Notice that $R^1\pi_*(\cF(-2))^\vee
=R^1\pi_*(\cF^\vee(-1))\otimes\cO_{\mbb P^1}(a+b)$ holds because 
$\pi_*(\cF(-2))=R^2\pi_*(\cF(-2))=0$. This is unclear for $\cF(l)$, $l\les-3$.) 
We deduce that $R^1\pi_*(\cal K(-2))=0$. 

\nit\textit{Step 5}\quad
Finally, the middle vertical sequence implies $R^1\pi_*(\cal M(-2))=0$. 
\end{proof}

\begin{lemma}\label{lm:M2}
Let $\cF$ be a vector bundle on $Y$ satisfying \eqref{eq:restr} -- \eqref{eq:Fv}. 
Then $\pi_*(\cal M\otimes\cO_\pi(l))$ and $R^2\pi_*(\cal M\otimes\cO_\pi(l))$ 
are locally free for all $l\in\mbb Z$. 
\end{lemma}

\begin{proof}
For $l\ges -1$, the last column and the middle line of \eqref{eq:M}, imply that 
$R^2\pi_*(\cal M(l))=0$. Since $R^1\pi_*(\cal M(l))=0$, it follows that 
$\pi_*(\cal M(l))$ is locally free. 

On the other hand, for $l\les -2$, the first line and the middle column in \eqref{eq:M} 
imply that $\pi_*(\cal M(l))=0$, so $R^2\pi_*(\cal M(l))$ is locally free again. 
\end{proof}

\begin{proof}(of theorem \ref{thm:p2-p1})
All that remains to prove is that $\cal M\cong\cO_Y^{\oplus r+2n}$. 
We do this in two steps.\smallskip  

\nit\textit{Step 1}\quad 
First we prove that $\cal M\cong\pi^*\pi_*\cal M$. 
Indeed, \cite[Ch. III, Theorem 12.11]{ha} yields  
$$
\biggl.
\begin{array}{l}
{R^2\pi_*(\cal M(l))}_x\to H^2(\phi_x,\cal M(l))
\text{ is surjective}
\\ 
R^2\pi_*(\cal M(l))\text{ is locally free}
\end{array}
\biggr\}
\Rightarrow
\begin{array}{l}
{R^1\pi_*(\cal M(l))}_x\to H^1(\phi_x,\cal M(l))
\\ 
\text{is an isomorphism, }\forall\,x\in\mbb P^1,
\end{array}
$$
thus $H^1(\phi_x,\cal M(l))=0$ for all $x\in\mbb P^1$ and $l\in\mbb Z$. 
Horrocks' criterion \cite[Lemma 1, pp.334]{bh} implies that the restriction 
of $\cal M$ to each fibre of $\pi$ splits into a direct sum of line bundles. 

The restriction of \eqref{eq:M} to  $\fbr$ is 
the monad \eqref{eq:O}, so $\cal M_\fbr\cong\cO_\fbr^{\oplus r+2n}$. 
Then $\cal M(-1)$ splits fibrewise, and its direct image under 
$\pi$ vanishes. (It is simultaneously a torsion and torsion-free sheaf.) 
As before, $\,{\pi_*(\cal M(-1))}_x\to\Gamma(\phi_x,\cal M(-1))\,$ 
is an isomorphism, so $\Gamma(\phi_x,\cal M(-1))=0$ and the degrees 
of the direct summands of $\cal M_{\phi_x}$ are less or equal to zero. 
As the (total) degree of $\cal M$ is zero, we have 
$\cal M_{\phi_x}\cong\cO_{\phi_x}^{\oplus r+2n}$, for all $x\in\mbb P^1$, 
so the natural homomorphism $\pi^*\pi_*\cal M\to\cal M$ is an isomorphism. 

\nit\textit{Step 2}\quad Let us denote $\cal S:=\pi_*\cal M$. The restriction 
of \eqref{eq:M} to the exceptional line $\L\cong\mbb P^1$ is a monad over $\mbb P^1$, 
whose middle entry is $\cal S$ and all the other entries are trivial vector bundles. 
It follows that $\cal S$ is itself trivial. This finishes the proof of the existence of 
the monad \eqref{eq:AOB}. 

Now we assume that $\cF_\fbr$ is stable. Then \cite[Remark, pp. 332]{bh} implies 
that the restriction of the monad to $\fbr$, whose cohomology is $\cF_\fbr$, 
is uniquely defined up to the $G$-action. Thus the same statement holds for 
$\cF$. Furthermore, an element $(g_1,g,g_2)\in G$ in the isotropy group of the 
action induces an automorphism of $\cF$. As $\cF_\fbr$ is stable, this 
automorphism is the multiplication by a scalar $\veps$, so 
$(g_1,g,g_2)=(\veps\bone_r,\veps\bone_{r+2n},\veps\bone_r)$. 
\end{proof}

Let $\textbf{H}$ be the affine space underlying 
$\Hom(\mbb C^{n},\mbb C^{r+2n}){\otimes}\,\Gamma(\cO_\pi(1))$, and define: 
\begin{eqnarray}
\bar{\bf V}{:=}\,
\bigl\{({\bf A},^{\sf t}\kern-1pt{\bf B})\,{\in}\textbf{H}^2
\,\bigl|\,
{\bf A}\text{ injective},\kern1ex{\bf B}\text{ surjective, }
{\bf B}{\bf A}=0
\bigr.\bigr\},
\label{eq:bV} 
\\
{\bf V}{:=}\,
\bigl\{({\bf A},^{\sf t}\kern-1pt{\bf B})\,{\in}\bar{\bf V}
\,\bigl|\,
{\bf A}_y\text{ injective,}\,\forall\,y\in Y
\bigr.\bigr\}.
\label{eq:V}
\end{eqnarray}
The group $G=[\Gl(n){\times}\Gl(r+2n){\times}\Gl(n)]/\mbb C^*$ 
acts on the affine variety 
$\{({\bf A},^{\sf t}\kern-1pt{\bf B})\,{\in}\textbf{H}^2\,\bigl|\,{\bf B}{\bf A}=0\}$, 
and $\bar{\bf V},\bf V$ are $G$-invariant open subsets.

\begin{corollary}
For $n\,{\ges}\,r$ and $c>r(r-1)n(a+b)$, the moduli space 
\\ \centerline{$\bar M_Y^\vb{=}\bar M_{Y_{a,b}}^\vb{=}
\bar M_{Y_{a,b}}^{L_c}\bigl(r;0,n[\cO_\pi(1)]^2,0\bigr)^\vb$} 
of $L_c$-semi-stable vector bundles on $\kern-1ptY_{a,b}\kern-1pt$ satisfying 
\eqref{eq:restr}\kern0.5pt--\kern1.5pt\eqref{eq:Fv} is the quotient 
of an open subset\footnote{Indeed, the extensions \eqref{eq:extq},\eqref{eq:extk} 
satisfy \eqref{eq:bone}.} of $\bf V$ by the action \eqref{eq:G} of $G$. 
For $\cF\in\bar M_Y^\vb$, holds 
\begin{equation}\label{eq:m}
\chi(End(\cF))=1-m,\quad\text{with}\quad m:=2(1+a+b)nr-r^2+1.
\end{equation}
In particular, if $\bar M_Y^\vb$ is non-empty, then its dimension is at least $m$. 
\end{corollary}

\begin{proof}
The condition $\bf{BA}=0$ imposes at most $n^2\cdot h^0(\cO_\pi(2))$ conditions, 
so the dimension of $\bar M_Y^\vb$ is at least \\ 
\centerline{$
2(3+a+b)n(r+2n)-(6+4a+4b)n^2-[2n^2+(r+2n)^2-1]
\srel{\eqref{eq:m}}=m.
$}
The Euler characteristic $\chi(End(\cF))$ is given by Riemann-Roch. 
\end{proof}

The monad \eqref{eq:AOB} still makes sense for 
$(\bf A,^{\sf t}\kern-1pt\bf B)\in\bar{\bf V}$. Its cohomology 
is a sheaf on $Y$ which still satisfies \eqref{eq:chern} -- \eqref{eq:Fv}. These objects 
naturally occur if one wishes to compactify $\bar M_Y^\vb$. It is unlikely, however, 
that by adding these sheaves one gets a projective (complete) variety. 
(If one thinks off the group action on $\bar{\textbf{V}}$ in terms of quivers, 
one obtains a quiver with loops.)

\subsection{Geometric properties of $\bar M_{Y_{a,b}}^\vb$} 
At this point is natural to ask whether $\bar M_Y^\vb$ is non-empty, irreducible, 
and has the expected dimension. The irreducibility is a complicated issue. 
If our benchmark is the case $a=0,b=1,r=2$ (so $Y_{0,1}$ is the blow-up of 
$\mbb P^3$ along a line), the question reduces to the irreducibility of the moduli 
space of rank two mathematical instantons (see \cite{ti,ctt} for details). 
The recent answer to this problem (see \cite{ti1,ti2}) involves impressive computations. 

For this reason, our approach is similar to \cite{ti}, namely we pin down 
a `main' component of $\bar M_Y^\vb$ which has all the desired properties.  
For shorthand, we let $\cO(k,l)\,{:=}\cO_\pi(k)\otimes\pi^*\cO_{\mbb P^1}(l)$. 
The following remark will be useful:\footnote{The computations are unpleasant, 
and the author used MAPLE. 
For $u,v$ as in \eqref{eq:h2y} and $\rk(\cal S)=s$, holds:\\ 
${\rm Td}(Y_{a,b})=1+\frac{3u-(a+b-2)v}{2}+u^2-\frac{(4(a+b)-9)uv}{6}+u^2v$, \\ 
${\rm ch}(\cal S(-1,-1))=s-s(u+v)+\frac{s+c}{2}u^2+suv
-\Big(\frac{s+c}{2}+\frac{(s+3c)(a+b)}{6}\Big)u^2v$.} 
by the Riemann-Roch formula, 
\begin{equation}\label{eq:-1-1}
\chi(\cal S(-1,-1))=0,\;\text{for any sheaf $\cal S$ on $Y$ with $c_1=0$, $c_2=cu^2$}. 
\end{equation}
In particular, for $\cF\in\bar M_Y^\vb$ and $\cal S=End(\cF)$, we deduce 
\\ \centerline{
$H^1\big(\,End(\cF)(-1,-1)\,\big)=0\;\Leftrightarrow\;H^2\big(\,End(\cF)(-1,-1)\,\big)=0.$
} 

\begin{theorem}\label{thm:MM}
Let $D\in|\cO_\pi(1)|$ be a generic section, 
so $D\cong\mbb P(\cO_{\mbb P^1}(-a)\oplus\cO_{\mbb P^1}(-b))$, and consider 
the following `main component' of $\bar M_Y^\vb$: 
\begin{equation}\label{eq:MM}
\mbb M:=\{\cF\in\bar M_Y^\vb\mid H^1\big(\,End(\cF)(-1,-1)\,\big)=0,
\;\cF_D\text{ is $L_c$-semi-stable}\,\}.
\end{equation}
Then $\mbb M$ is non-empty, irreducible, generically smooth of the expected dimension, 
and the locus corresponding to the stable vector bundles is dense. 
Moreover, $\mbb M$ is a rational variety. 
\end{theorem}
\nit The proof of this statement is contained in the forthcoming lemmas. 

For $\cF\in\mbb M$ and general $D\in|\cO_\pi(1)|$ and $P\in|\pi^*\cO_{\mbb P^1}(1)|$, 
the restrictions $\cF_D$ and $\cF_P$ are semi-stable, and theorem \ref{thm:main} 
implies that $\cF$ is trivializable along $\l:=D\cap P$ 
(so $\cF$ automatically satisfies the technical condition \eqref{eq:restr}). 

\begin{definition}\label{defn:m-frame}
For $D,P$ as above, let $\l:=D\cap P$ and  $\Delta:=D\cup P$. We denote 
\\ \centerline{
$\bar M_{D,\l}^\vb:=\bar M_{D,\l}^{L_c}(r;0,n(a+b))^\vb$\quad 
(respectively $\bar M_{P,\l}^\vb, \bar M_{\Delta,\l}^\vb$) 
} 
the moduli spaces of semi-stable vector bundles on $D$ (respectively on $P,\Delta$), 
framed along $\l$. (See remark \ref{def:frame} for the definition of a framing.) 
\end{definition}
Then the map which identifies (glues) the framings  
\begin{equation}\label{eq:glue}
\bar M_{D,\l}^\vb\times \bar M_{P,\l}^\vb\to\bar M_{\Delta,\l}^\vb
\end{equation}
is an isomorphism (its inverse is the restriction to $D,P$), and ${\PGl}(r)$ 
acts on $\bar M_{\Delta,\l}^\vb$ by changing the framing along $\l$. 
The quotient map for this action is the morphism which forgets the framing. 
We denote $\bar M_{\Delta}^\vb:=
\bar M_{\Delta,\l}^\vb/{\PGl}(r)$. 
A key role for understanding the geometry of $\mbb M$ is played by the rational map 
\begin{equation}\label{eq:th}
\Th:\mbb M\dashto \bar M_{\Delta}^\vb,\quad  \cF\to \cF_\Delta.
\end{equation}

\begin{lemma}\label{lm:MD}
Let $M_\Delta^\vb$ be the open locus of vector bundles whose restrictions to $D,P$ 
are stable. Then $M_\Delta^\vb$ is birational to 
$M_D^{L_c}(r;0,n)^\vb\times M_{P,\l}^\vb$, thus $M_\Delta^\vb$ is a rational variety 
of dimension 
\begin{equation}\label{eq:dimD}
\dim\bar M_\Delta^\vb=2(a+b)nr+2nr-(r^2-1)=m.
\end{equation}
\end{lemma}

\begin{proof}
First, $D$ is isomorphic to the Hirzebruch surface $Y_{\ell}$, with $\ell=b-a$, 
so $M_{D,\l}^\vb$ is birational to $M_D^{L_c}(r;0,n)^\vb\times\PGl(r)$, 
according to corollary \ref{cor:rtl-frameY}. Thus, by the definition, $M_\Delta^\vb$ 
is birational to $M_D^{L_c}(r;0,n)^\vb\times M_{P,\l}^\vb$. 
Second, $M_{P,\l}^\vb$ is irreducible because $P\cong\mbb P^2$ 
(see \cite[Theorem 2.2]{hu}), and is also rational, by \ref{cor:P2}. 
\end{proof}


\subsubsection{Differential properties of $\mbb M$} 
We start by addressing the generic smoothness of $\mbb M$. 

\begin{lemma}\label{lm:h2=0}
For all $\cF\in\mbb M$ holds: 
\begin{enumerate}
\item $H^2(End(\cF))=0$;
\item The differential of $\Th$ at $\cF$ is an isomorphism;
\item Each irreducible component has the expected dimension, the locus corresponding 
to stable bundles in dense, and $\Th$ is generically finite onto $\bar M_{\Delta}^\vb$.
\end{enumerate}
\end{lemma}

\begin{proof}
(i) Since $\cF_D$ and $\cF_P$ are semi-stable, $\cal E:=End(\cF)$ is the same, 
thus $\cal E_D,\cal E_P(-1,0)$ have vanishing $H^2$. 
Now we apply this in 
$\cal E(-1,-1)\subset\cal E(-1,0)\surj\,\cal E_P(-1,0)$ and 
$\cal E(-1,0)\subset\cal E\surj\,\cal E_D$.

\nit(ii) The differential of $\Th$ at $\cF$ is the restriction homomorphism 
$H^1(\cal E)\to H^1(\cal E_\Delta)$. The exact sequence 
$\cal E(-1,-1)\subset \cal E\surj\,\cal E_\Delta$ implies that this is indeed an isomorphism. 

\nit(iii) Since $\rd\Th$ is an isomorphism, the restriction of $\Th$ to each component 
of $\mbb M$ is dominant. But the stable vector bundles are dense in both 
$\bar M_D^\vb, \bar M_P^\vb$ (see theorem \ref{thm:hirz}(iv)), and $\cF$ is stable 
on $Y$, as soon as its restriction to $\cF_\Delta$ is stable. (This also shows that $\Th$ 
is well-defined at the generic point of each irreducible component of $\mbb M$.) 
In this case, $\ext^1(\cF,\cF)=m$, so each component has the expected dimension. 
For the generic finiteness of $\Th$, use \eqref{eq:dimD}.
\end{proof}

\subsubsection{Non-emptiness of $\mbb M$} 
Here we give explicit examples of vector bundles satisfying the defining properties 
of $\mbb M$. 

\begin{lemma}\label{lm:Ft}
Assume $r\ges 3$. There is a non-empty component $\mbb M_o\subset\bar M_Y^\vb$ 
such that the generic $\cF\in\mbb M_o$ has the following properties: 
\begin{enumerate}
\item 
Its restriction to the generic divisor in $|\cO_\pi(1)|$ is semi-stable, the restriction to 
the line cut out by two generic divisors in $|\cO_\pi(1)|$ is trivializable, and is semi-stable 
on \emph{all} the fibres of $\pi$. Hence $\cF$ is $L_c$-semi-stable on $Y$. 
\item 
$H^1\big(End(\cF)(-1,-1)\big)=0$.
\end{enumerate}
\end{lemma}

\begin{proof}
\nit\textit{Step 1}\quad 
We consider two generic divisors $D,D'\in|\cO_\pi(1)|$, and the (fat) line 
$L_n:=nD\cap D'$ in $Y$. Its ideal sheaf admits the resolution 
$0\,{\to}\,\cO_\pi(-n-1)\,{\srel{a}{\to}}\,
\cO_\pi(-1)\,{\oplus}\,\cO_\pi(-n)\,{\to}\,\cal I_{L_n}\,{\to}\,0,$ 
and determines the exact sequence 
\begin{equation}\label{eq:F0}
0\to \cO_\pi(-n-1)\srel{(0,a)}{\lar}
\cO_Y^{r-1}\oplus\cO_\pi(-1)\oplus\cO_\pi(-n)\to
\cF_0:=\cO_Y^{r-1}\oplus\cal I_{L_n}\to0.
\end{equation}
Then $\cF_0$ is torsion free of rank $r$, with $c_1=0,c_2=n\cdot [\cO_\pi(1)]^2$. 
Its restriction to the generic intersection of two divisors in $|\cO_\pi(1)|$ is trivializable, 
so $\cF_0$ is $L_c$-semi-stable, for all $c>0$. 

\nit\textit{Step 2}\quad 
By deforming $0$ to $t\in\Gamma(\cO_\pi(n+1))$ in \eqref{eq:F0}, we obtain a flat 
family of sheaves (the Hilbert polynomial is constant) on $Y$ 
\begin{equation}\label{eq:Ft}
0\to \cO_\pi(-n-1)\srel{(t,a)}{\lar}
\cO_Y^{r-1}\oplus\cO_\pi(-1)\oplus\cO_\pi(-n)\to
\cF_t\to0
\end{equation}
Also, since $r\ges 3$, $(t,a)$ is pointwise injective for generic $t$,  
so $\cF_t$ is locally free. 

\nit\textit{Claim}\quad For generic $t$, 
the vector bundle $\cF_t$ defined by \eqref{eq:Ft} has the desired properties. \\ 
-- $\cF_0$ satisfies \eqref{eq:restr}  and \eqref{eq:restrL}, 
so the same holds for generic $t$.\\ 
--  One may check that $\cF_0$ satisfies \eqref{eq:R2}, so the same holds for $\cF_t$. 
Alternatively, $\cF_t$ is $\pi$-fibrewise semi-stable, so \eqref{eq:R2} is automatically 
satisfied.\\ 
-- \eqref{eq:F} and \eqref{eq:Fv} follow directly from \eqref{eq:Ft}. \\ 
-- The three properties at (i) are open in flat families of torsion free sheaves, 
and they hold for $\cF_0$. Thus the same holds for generic $t$. \smallskip

\nit-- Let us verify (ii). (Incidentally, observe that $\Ext^2(\cF_0,\cF_0)\neq 0$.) 
Since $\cF_t$, so $End(\cF_t)$ is semi-stable, we have 
$h^0\big(End(\cF_t)(-1,-1)\big){=}\,h^3\big(End(\cF_t)(-1,-1)\big){=}\,0$. 
By \eqref{eq:-1-1}, holds 
\\ \centerline{
$h^1(End(\cF_t)(-1,-1))=0\;\Leftrightarrow\;h^2(End(\cF_t)(-1,-1))=0$. 
} 
This latter property is easier to check. 
For generic $t$, the dual of \eqref{eq:Ft} yields 
\\ \centerline{ 
$0\to End(\cF_t)(-1,-1)\to\cF_t(-1,-1)^{r-1}\oplus\cF_t(0,-1)\oplus\cF_t(n-1,-1)
\to\cF_t(n,-1)\to0.$ 
}
Now remark that \eqref{eq:Ft} implies $H^1(\cF_t(n,-1))\,{=}\,0$. 
Second, we claim that $H^2(\cF_t(k,-1))=0$, for all $k\ges -1$. 
Indeed, the vanishing holds for $k=-1$, by \ref{lm:coh}(iii). For the induction, 
let $L$ be the intersection of two generic divisors in $|\cO_\pi(1)|$, twist the exact 
sequence $\cO_\pi(-2)\,{\subset}\,\cO_\pi(-1)^{\oplus 2}\,{\surj}\,\cal I_L$ 
by $\cF_t(k,-1)$, and use the semi-stability of $\cF_t$. 

We explicitly produced vector bundles satisfying the lemma. The properties are open 
in flat families, so there is a non-empty component $\mbb M_o\subset\bar M_Y^\vb$ 
which (generically) satisfies all of them. 
\end{proof}

It remains to address the case of rank two vector bundles. 

\begin{lemma}\label{lm:Ft2}
For $r=2$, there is a non-empty component $\mbb M_o\subset\mbb M$ whose 
generic point satisfies \ref{lm:Ft} {\rm(i)},{\rm(ii)}.
\end{lemma}

\begin{proof}
Let $Z$ be the union of $n$ sections of $\pi:Y\to\mbb P^1$, such that each 
$L\subset Z$ is the intersection of two generic divisors in $|\cO_\pi(1)|$. 
The Hartshorne-Serre correspondence (see \cite{ha1,arr}) yields a vector bundle $\cF$ 
with $c_1=0,c_2=[Z]=n\cdot[\cO_\pi(1)]^2$, which fits into an exact sequence 
\begin{equation}\label{eq:IZ}
0\to\cO_\pi(-1)\to\cF\to\cal I_Z(1)\to 0;\quad\cF_Z\cong\cO_Z^{\,2}.
\end{equation}
The properties (i) are easy to verify. For the second statement, the dual of \eqref{eq:IZ} 
yields 
\\ \centerline{ 
$0\srel{\ref{lm:coh}\text{(iv)}}{=}H^1(\cF(-2,-1))\to H^1(End(\cF)(-1,-1))
\to H^1(\cal I_Z\otimes\cF(0,-1))\srel{\phi}{\to}H^2(\cF(-2,-1))...$ 
} 
We claim that $\phi$ is an isomorphism. Indeed, $\cF$ is obtained by glueing 
the local Koszul resolutions of the components $\l$ of $Z$. Thus $\phi$ is the 
boundary map corresponding to the Koszul resolution of (any) one of $L\subset Z$. 
Then $\cF_Z\cong\cO_Z^{\,2}$ implies 
$H^1(\cal I_Z\otimes\cF(0,-1))=H^1(\cal I_L\otimes\cF(0,-1))=H^1(\cF(0,-1))$. 
Now use $\cF(-2,-1)\subset\cF(-1,-1)^2\to\cal I_L\otimes\cF(0,-1)$ to deduce that 
$H^1(\cal I_L\otimes\cF(0,-1))\to H^1(\cF(-2,-1))$ is an isomorphism. 
\end{proof}

\subsubsection{Irreducibility of $\mbb M$} 
We are going to prove that $\mbb M=\mbb M_o$, which yields the conclusion. 

\begin{lemma}\label{lm:inj}
For $\cF_o\in\mbb M_o$ and $\cF\in\mbb M$ arbitrary, 
holds $H^1(Hom(\cF_o,\cF)(-1,-1))=0$. 
\end{lemma}

\begin{proof}
First notice that $\cF_o^\vee$ satisfies the hypotheses of the theorem \ref{thm:p2-p1}: 
the conditions \eqref{eq:F}, \eqref{eq:Fv} are satisfied by \eqref{eq:-1-1}, and 
\eqref{eq:R2} holds because $\cF_o$ is semi-stable on all the fibres of $\pi$. 
Thus $\cF_o^\vee$ is the cohomology of a monad \eqref{eq:AOB}, and we denote by 
$\cal Q_o$ the corresponding entry in its display. 

For $\cal H:=Hom(\cF_o,\cF)$, we prove that $H^1(\cal H(-1,-1))=0$. 
Since $\cF$ is semi-stable, the exact sequence 
\\ \centerline{
$\begin{array}{c}
0\to\cal H(-1,-1)\to\cal Q_o\otimes\cF(-1,-1)\to\cF(0,-1)^n\to 0,
\end{array}$
}
yields 
\begin{equation}\label{eq:QF}
0=\Gamma(\cF(0,-1))\to 
H^1(\cal H(-1,-1))\to H^1(\cal Q_o\otimes\cF(-1,-1))\to H^1(\cF(0,-1))^n.
\end{equation}
The conclusion follows as soon as we prove that the rightmost arrow is an isomorphism. 
For this, we must understand $\cal Q_o$ better. 

As $\cF_{o,\l}\cong\cal O_\l^r$, the restriction to $\l$ of the monad defining 
$\cF_{o}^\vee$ yields $\cal O_\l^{\,r}\subset\cal Q_{o,\l}\surj\,\cal O_\l(1)^n$. This 
extension is necessarily trivial, so $\cal Q_{o,\l}\cong\cO_\l^{\,r}\oplus\cO_\l(1)^n$. 
We deduce that the sheaf homomorphism $s$ in the diagram (I) below is injective: 
\begin{equation}\label{eq:si}
\scalebox{.8}{
$\begin{array}{cc|ccc|cc}
\xymatrix@R=1.75em@C=1.5em{
&\kern2ex\cO_Y^r\ar@{=}[r]\ar@{^(->}[d]&\cO_Y^r\ar@{^(->}@<-1ex>[d]^-s
\\ 
\cO_\pi(-1)^n\ar@{^(->}[r]\ar@{=}[d]
&\cO_Y^{r+2n}\ar@{->>}[d]\ar@{->>}[r]
&\cal Q_o\ar@{->>}@<-1ex>[d]
\\
\cO_\pi(-1)^n\ar@{^(->}[r]&\cO_Y^{2n}\ar@{->>}[r]&
\cal R_o
}
&&&
\xymatrix@R=1.75em@C=1.5em{
&\cO_Y^n\ar@{=}[r]\ar@<-1ex>@{^(->}[d]&\cO_Y^n\ar@<-1ex>@{^(->}[d]
\\ 
\cO_\pi(-1)^n\ar@{^(->}[r]\ar@{=}[d]
&\cO_Y^{2n}\ar@<-1ex>@{->>}[d]\ar@{->>}[r]
&\cal R_o\ar@<-1ex>@{->>}[d]
\\ 
\cO_\pi(-1)^n\ar@{^(->}[r]
&\cO_Y^n\ar@{->>}[r]&\cal S_o
}
&&&
\xymatrix@R=1.75em@C=1.5em{
\cO_\pi(-1)\ar@{^(->}[r]\ar@{^(->}[d]
&
\cO_Y^n\ar[r]\ar@<-1ex>@{^(->}[d]&\cO_D\ar@<-1ex>@{^(->}[d]
\\ 
\cO_\pi(-1)^n\ar@{^(->}[r]\ar@{=}[d]
&\cO_Y^{n}\ar@<-1ex>@{->>}[d]\ar@{->>}[r]
&\cal S_o\ar@<-1ex>@{->>}[d]
\\ 
\cO_\pi(-1)^{n-1}\ar@{^(->}[r]
&\cO_Y^{n-1}\ar@{->>}[r]&\cal S_o^{(n-1)}.
}
\\[5mm] 
\text{(I)}&&&\text{(II)}&&&\text{(III)}
\end{array}
$}
\end{equation}
Now we further decompose $\cal R_o$; clearly, there is a decomposition 
$\cO_Y^{2n}=\cO_Y^n\oplus\cO_Y^n$, such that the diagram (II) 
has exact columns and rows; it determines the torsion sheaf $\cal S_o$. 

In order to understand this latter, we proceed inductively: clearly, there is a decomposition 
$\cO_Y^n=\cO_Y\oplus\cO_Y^{n-1}$ such that the diagram (III) has exact rows 
and columns. Here $D$ stands for a generic (the starting $\cF_o$ is so) divisor in 
$|\cO_\pi(1)|$. By repeating $n$ times this process, we deduce that $\cal S_o$ is a 
successive extension of $\cO_{D_j}$, with $D_j\in|\cO_\pi(1)|$ for $j=1,\ldots,n$. 
Let $\cal S_o^{(\nu)}$ be the sheaf obtained by $\nu$ extensions, $\nu=1,\ldots,n$. 

By using that $H^1(\cF(-2,-1))=H^1(\cF(-1,-1))=H^2(\cF(-1,-1))=0$ 
(see \ref{lm:coh}), we deduce the following implications, for arbitrary $D\in|\cO_\pi(1)|$: 
\\ \centerline{
$\begin{array}{rl}
\text{(I)}\Rightarrow&
H^1(\cal Q_o\otimes\cF(-1,-1))\srel{\cong}{\to} H^1(\cal R_o\otimes\cF(-1,-1))^n,
\\[1.5ex] 
\text{(II)}\Rightarrow&
H^1(\cal R_o\otimes\cF(-1,-1))\srel{\cong}{\to} H^1(\cal S_o\otimes\cF(-1,-1))^n,
\\[1.5ex]
\cO_\pi(-1)\subset\cO_Y\surj\,\cO_D\;\Rightarrow&
\Gamma(\cF(-1,-1)_D)=H^2(\cF(-1,-1)_D)=0,
\\ 
&
H^1(\cF(-1,-1)_D)\srel{\cong}{\to}H^2(\cF(-2,-1)),
\\[1.5ex]
\text{use inductively (III)}\Rightarrow&
\Gamma(\cal S_o^{(\nu)}\otimes\cF(-1,-1))=0,\;\nu=1,\ldots,n,
\\[.5ex] 
&\kern-25mm
0\to H^1(\cF(-1,-1)_D)\to H^1(\cal S_o^{(\nu)}\otimes\cF(-1,-1))
\to H^1(\cal S_o^{(\nu-1)}\otimes\cF(-1,-1))\to0.
\end{array}$
}
Overall, we deduce that $H^1(\cal Q_o\otimes\cF(-1,-1))\srel{\cong}{\to}
\ouset{j=1}{n}H^1(\cF(-1,-1)_{D_j})\srel{\cong}{\to}H^2(\cF(-2,-1))^n$. 
Thus the rightmost arrow in \eqref{eq:QF} is an isomorphism, and consequently 
$H^1(\cal H(-1,-1))=0$. 
\end{proof}

\begin{lemma}\label{lm:irred-M}
$\mbb M$ is irreducible.
\end{lemma}

\begin{proof}
We constructed the component $\mbb M_o$, and assume $\mbb M'$ is 
another component of $\mbb M$. Since both $\mbb M_o,\mbb M'$ dominate 
$\bar M_{\Delta}^\vb$ (see \ref{lm:h2=0}), we find general points 
$\cF_o\in\mbb M_o$ and $\cF\in\mbb M'$ with the properties: 

-- $\Th(\cF)=\Th(\cF_o)$, that is $\cF_\Delta\cong\cF_{o,\Delta}$. 
 
-- Both $\cF,\cF_o$ are stable (by the density of the stable locus). 

-- $\cF_o$ satisfies the conditions \ref{lm:Ft}(i).

Then, for $\cal H:=Hom(\cF_o,\cF)$, the exact sequence 
\\ \centerline{
$0\to\Gamma(\cal H(-\Delta))\to\Gamma(\cal H)\to\Gamma(\cal H_\Delta)
\to H^1(\cal H(-\Delta))\to\ldots,$
}
has vanishing left hand side (as $\cal H$ is semi-stable) and also right hand side 
(by lemma \ref{lm:inj}). Hence the isomorphism $\cF_{o,\Delta}\to\cF_\Delta$ 
can be extended to a homomorphism $\cF_o\to\cF$. Its determinant is a section 
of $\cO_Y$, non-zero along $\Delta$, so $\cF_o,\cF$ are isomorphic too. 
We deduce that $\mbb M_o=\mbb M'$, that is $\mbb M$ is irreducible, 
since they are irreducible components, and their general points coincide. 
\end{proof}

\subsubsection{Rationality of $\mbb M$} 

\begin{lemma}\label{lm:M}
$\Th:\mbb M\to\bar M_{\Delta}^\vb$ is generically injective and birational, 
so $\mbb M$ is a rational variety. 
\end{lemma}

\begin{proof} 
The same argument as in the proof of the previous lemma shows that 
$\Th:\mbb M\to\bar M_\Delta^\vb$ is generically injective. 
Let $\mbb M^s\subset\mbb M_o$ the locus consisting of vector bundles whose 
restrictions to both $D,P$ are stable. Then $\Th|_{\mbb M^s}$ is well-defined, 
generically injective, and dominant. Zariski's main theorem implies that $\Th$ 
is birational, so $\mbb M$ is a rational variety by \ref{lm:MD}. 
\end{proof}

\begin{remark}\label{rmk:frame}
\nit{\rm(i)} 
In several cases (see \cite{do,bbr}) is more convenient to work with framed vector 
bundles (especially for the existence of universal families). We can reformulate the 
theorem by saying that the moduli space $\mbb M_{\l}$, consisting of semi-stable 
vector bundles $\cF\in\mbb M$ together with a framing along the line $\l$ 
(in contrast with the usual framings along divisors) is birational to 
$\bar M_{\Delta,\l}^\vb$. 

\nit{\rm(ii)} 
For $a=0, b=1$, one may easily check that $Y_{0,1}$ is isomorphic to the blow-up 
of $\mbb P^3$ along a line, and theorem \ref{thm:p2-p1} reduces precisely to the monad 
construction  \cite{bh,do} of instantons on $\mbb P^3$ trivialized along the line. 

\nit{\rm(iii)} 
We conclude by noticing that  theorem \ref{thm:p2-p1} yields also principal symplectic, 
respectively orthogonal bundles on $Y_{a,b}$. (Higher rank symplectic instanton bundles 
on $\mbb P^3$ have been constructed recently in \cite{bmt}.) In this case, 
\eqref{eq:F} and \eqref{eq:Fv} are equivalent by the Riemann-Roch formula, 
so one should impose only the conditions \eqref{eq:restr} -- \eqref{eq:F}. 
The outcome is that there is a non-degenerate (skew-)symmetric, bilinear form $b$ on 
$\mbb C^{r+2n}$ (the middle term of \eqref{eq:AOB}), such that the homomorphisms 
$\bf{A,B}$ are dual to each other with respect to $b$, that is 
$\textbf{B}=\,^{\sf t}\kern-1pt{\bf A}\cdot b$ (compare with \cite[Section 4]{bh}). 
The monad condition $\textbf{BA}=0$ translates into 
$^{\sf t}\kern-1pt{\bf A}\cdot b\cdot{\bf A}=0$. 
The properties of the corresponding moduli spaces will be investigated in a future article.
\end{remark}


\end{document}